\documentclass[12pt,francais]{smfart}

\usepackage[dvips]{epsfig}

\usepackage[francais]{babel}

\usepackage{smfthm}

\usepackage{mathrsfs}

\usepackage[dvips,all]{xy}

\usepackage[applemac]{inputenc}

\usepackage{hyperref}
\hypersetup{
    bookmarks=true,         
    unicode=false,          
    pdftoolbar=true,        
    pdfmenubar=true,        
    pdffitwindow=false,     
    pdfstartview={FitH},    
    pdftitle={My title},    
    pdfauthor={Author},     
    pdfsubject={Subject},   
    pdfcreator={Creator},   
    pdfproducer={Producer}, 
    pdfkeywords={keyword1} {key2} {key3}, 
    pdfnewwindow=true,      
    colorlinks=false,       
    linkcolor=red,          
    citecolor=green,        
    filecolor=magenta,      
    urlcolor=cyan           
}

\renewcommand{\L}{\mathbb L}

\newtheorem{theorem}{Th\'eor\`eme}[section]
\newtheorem{lemma}[theorem]{Lemme}
\newtheorem{proposition}[theorem]{Proposition}
\newtheorem{corollary}[theorem]{Corollaire}
\newtheorem{rappel}[theorem]{}
\newtheorem{problem}[theorem]{Probl\`eme}
\newtheorem{propriete}[theorem]{Propri\'et\'e}

\theoremstyle{definition}
\newtheorem{definition}[theorem]{D\'efinition}
\newtheorem{remark}[theorem]{Remarque}
\newtheorem*{rem*}{Remarque}
\newtheorem*{thm*}{Théorème}

\newtheorem{notation}[theorem]{Notation}
\newtheorem{question}[theorem]{Question}
\newtheorem{example}[theorem]{Exemple}
\newtheorem{conjecture}[theorem]{Conjecture}
\newtheorem{exercice}[theorem]{Exercice}

\newcommand{\bi}{\begin{itemize}}
\newcommand{\ei}{\end{itemize}}
\newcommand{\be}{\begin{enumerate}}
\newcommand{\ee}{\end{enumerate}}

\newcommand{\bpf}{\begin{proof}}
\newcommand{\epf}{\end{proof}}

\newcommand{\bpro}{\begin{propriete}}
\newcommand{\epro}{\end{propriete}}

\newcommand{\bt}{\begin{theorem}}
\newcommand{\et}{\end{theorem}}
\newcommand{\brap}{\begin{rappel}}
\newcommand{\erap}{\end{rappel}}
\newcommand{\bnt}{\begin{notation}}
\newcommand{\ent}{\end{notation}}
\newcommand{\bd}{\begin{definition}}
\newcommand{\ed}{\end{definition}}
\newcommand{\ble}{\begin{lemma}}
\newcommand{\ele}{\end{lemma}}
\newcommand{\bpr}{\begin{proposition}}
\newcommand{\epr}{\end{proposition}}
\newcommand{\bre}{\begin{remark}}
\newcommand{\ere}{\end{remark}}
\newcommand{\bco}{\begin{corollary}}
\newcommand{\eco}{\end{corollary}}
\newcommand{\beq}{\begin{equation}}
\newcommand{\eeq}{\end{equation}}
\newcommand{\bq}{\begin{question}}
\newcommand{\eq}{\end{question}}
\newcommand{\bp}{\begin{problem}}
\newcommand{\ep}{\end{problem}}
\newcommand{\beqn}{\begin{eqnarray*}}
\newcommand{\eeqn}{\end{eqnarray*}}
\newcommand{\bex}{\begin{example}}
\newcommand{\eex}{\end{example}}
\newcommand{\ber}{\begin{exercice}}
\newcommand{\eer}{\end{exercice}}

\newcommand{\sk}{\smallskip}
\newcommand{\mk}{\medskip}
\newcommand{\bk}{\bigskip}

\newcommand{\bcj}{\begin{conjecture}}
\newcommand{\ecj}{\end{conjecture}}

\renewcommand{\O}{\Omega}

\title[Tissus alg\'Ebriques exceptionnels]{
 Tissus   alg\'Ebriques exceptionnels
}

\author[L. Pirio]{Luc Pirio\vspace{0.3cm}\\
\small{IRMAR,  UMR 6625 du CNRS,  Universit\'e Rennes 1}\\ \small{Campus de Beaulieu 35042 Rennes Cedex, France}
\vspace{0.15cm}\\
E-mail: {\it luc.pirio@univ-rennes1.fr}}
\address{IRMAR,  UMR 6625 du CNRS,  Universit\'e Rennes 1, Campus de Beaulieu, 35042 Rennes Cedex, France}
\email{luc.pirio@univ-rennes1.fr}

\begin{document}

\maketitle

\begin{abstract}
Dans \cite{PTweb}, nous avons montré que pour $r>1$, $n\geq 2$ et 
$d\geq (r+1)(n-1)+2$, 
 un $d$-tissu de type $(r,n)$ de rang maximal est  algébrisable au sens classique, sauf peut-être lorsque  $n\geq 3$ et $d= (r+2)(n-1)+1$. 
On s'int\'eresse ici à ce cas particulier.   
Sous ces hypothèses sur $n$ et $d$, on construit des exemples de ``{\it tissus alg\'ebriques exceptionnels}'': il s'agit de tissus algébriques d'incidence de rang maximal  qui ne  sont pas  algébrisables au sens classique.
\end{abstract}


Ce texte peut être vu comme une suite de \cite{PTweb}. On utilise aussi de façon cruciale plusieurs r\'esultats de   
\cite{PTgeom} et de \cite{piriorusso} qui concernent un probl\`eme de g\'eom\'etrie projective  
 intimement lié 
 au problème   considéré ici.  Il est donc naturel de se placer dans un cadre analytique complexe et c'est ce que nous ferons dans tout l'article. 
 
 \section{Introduction et présentation des résultats}
 
Nous commençons par introduire rapidement la théorie des tissus et par formuler le problème qui nous occupe dans cet article.  Le lecteur intéressé pourra consulter l'ouvrage de référence \cite{BB} (en allemand et un peu daté), l'article \cite{CG} qui contient une bonne exposition de la théorie ou encore le livre  plus récent \cite{PPbook}.  Ces références traitent principalement du cas de la codimension 1 mais le lecteur pourra y trouver des détails ainsi qu'une mise en perspective générale.
 Concernant les tissus de codimension quelconque, nous renvoyons à notre précédent article \cite{PTweb} et aux références qui s'y trouvent. 
\subsection{Tissus, relations abéliennes et rang}

\subsubsection{}
Par définition, un {\bf $\boldsymbol{d}$-tissu} $\mathcal T$ sur une variété complexe $X$ est  (localement) la donnée de $d$-feuilletages $\mathcal F_1,\ldots , \mathcal F_d$ sur $X$. Classiquement, on demande en plus que ces derniers soient de même codimension $r$ et en \og position générale\fg, c'est-à-dire que les espaces tangents des feuilles des $\mathcal F_i$  s'intersectent aussi transversalement que possible en un point générique de $X$.   On dit que $\mathcal T$ est de {\bf type} $\boldsymbol{(r,n)}$ lorsque la  codimension commune $r$ des ses feuilletages divise la dimension de l'espace ambiant, i.e. lorsque $\dim(X)=nr$ pour un certain entier $n\geq 2$.  Dans ce cas, l'hypothèse de position générale prend la forme précise suivante: en tout point $x\in X$ et  pour tout sous-ensemble $I\subset \{1,\ldots,d\}$ à $n$ éléments, les espaces tangents $T_x \mathcal F_{i}$, $i\in I$  sont en somme directe dans $T_xX$. \sk 

Deux tissus 
sont dit être {\bf équivalents} s'il existe un  biholomorphisme local qui envoie l'un sur l'autre. 
La {\bf géométrie des tissus} s'occupe de la classification des tissus modulo cette notion d'équivalence. 
\mk 

\subsubsection{}
\label{SS:TV}
L'exemple classique est celui du {\bf tissu algébrique} $\boldsymbol{\mathcal T_V}$ associé à une sous-variété algébrique projective réduite $V$ d'un espace projectif.  Si $V$ est de dimension $r\geq 1$ et de codimension $n-1\geq 1$, un sous-espace projectif générique $P$ de dimension $n-1$ de l'espace projectif ambiant $\mathbb P^{n+r-1}$ intersecte $V$ en $d=\deg(V)$ points distincts $p_1(P),\ldots,p_d(P)$. Par définition, ${\mathcal T_V}$ est le $d$-tissu sur la variété grassmannienne $G_{n-1}(\mathbb P^{n+r-1})$ dont les feuilles en $P$ sont les variétés de Schubert formées des sous-espaces de dimension $n$ qui passent par $p_i(P)$,  pour $i=1,\ldots,d$.  
En général, par exemple si  $V$ est supposée irréductible et 
non-dégénérée dans $\mathbb P^{n+r-1}$, $\mathcal T_V$ est de type $(r,n)$.
%
Un tissu est dit être {\bf  algébrisable} (au sens classique) s'il est équivalent à un tissu algébrique  ${\mathcal T_V}$ de cette sorte.
 
\sk

\subsubsection{}
L'un des attraits de la géométrie des tissus est qu'on peut la voir comme une sorte de généralisation de la géométrie algébrique projective classique (cf. \cite{C}). Par exemple, la notion de forme holomorphe se généralise aux tissus de la façon suivante: si un (germe de) tissu $\mathcal T=(\mathcal F_i)_{i=1}^d$ sur $(\mathbb C^{nr},0)$ est défini localement via de $d$ submersions $u_1,\ldots,u_d: (\mathbb C^{nr},0)\rightarrow (\mathbb C^{r},0)$, 
une {\bf relation abélienne} de $\mathcal T$ est un $d$-uplet $
(\omega_i)_{i=1}^d$
de $r$-formes différentielles holomorphes (nécessairement fermées) 
$\omega_i \in u_i^*\big(\Omega^r(\mathbb C^r,0)\big)$,  pour $i=1,\ldots,d$, 
telles que $\sum_{i=1}^d \omega_i=0$. 
 L'espace $\mathfrak A(\mathcal T)$ des relations abéliennes de $\mathcal T$ admet une structure naturelle d'espace vectoriel complexe dont la  dimension $\dim( \mathfrak A(\mathcal T))$, aussi 
 appelée le {\bf rang} de $\mathcal T$ et notée ${\rm rg}(\mathcal T)$,  est un invariant analytique du tissu.  \sk
 
\subsubsection{}
\label{SS:RAdeTV}
Soit $\mathcal T_V$ un tissu algébrique comme en {\it \ref{SS:TV}}, avec $V$ supposée lisse (pour simplifier).  Localement au voisinage de $P\in G_{n-1}(\mathbb P^{n+r-1})$, on a  $d$ applications $P'\mapsto p_i(P')$ telles que $P'\cdot V=\sum_{i=1}^d p_i(P')$ comme 0-cycle sur $V$. On vérifie que les $p_i$ sont des submersions qui définissent $\mathcal T_V$ près de $P$. Il découle alors du théorème d'addition d'Abel \cite{Griffiths} que pour toute $r$-forme holomorphe sur $V$, on a $\sum_{i=1}^d p_i^*(\omega)=0$ et  on peut donc interpréter  tout $d$-uplet $(p_i^*(\omega))_{i=1}^d$ comme une relation abélienne de $\mathcal T_V$.  On en déduit un morphisme linéaire injectif 
 $H^0(V,\Omega_V^r)\hookrightarrow \mathfrak A(\mathcal T_V)$ qui est en fait  surjectif, ce dernier point résultant du théorème d'Abel-Inverse (voir \cite{Griffiths} à nouveau). 
  Il en découle en particulier que 
  ${\rm rg}(\mathcal T_V)=h^0(\Omega_V^r)$: le rang du tissu algébrique associé à $V$ est égal au genre de cette dernière\footnote{Lorsque $V$ est singulière, il faut considérer les $r$-formes différentielles abéliennes sur $V$ pour avoir des résultats analogues, voir \cite{Griffiths,Barlet,HenkinPassare}.}.
  \subsubsection{}
  Un résultat important du domaine, mais assez facile au fond (voir \cite[\S2.1]{PTweb} pour une preuve élémentaire),  est l'existence de majorations explicites et optimales sur le rang. 
Si $\mathcal T$ est un $d$-tissu de type $(r,n)$, on doit  à Bol (lorsque $r=1$ et $n=2,3$), à Chern (pour $r=1$ et $n\geq 2$ quelconque) et enfin à Chern et Griffiths \cite{CGborne},  la majoration 
\begin{equation*}
\label{E:rg-rho}
{\rm rg}(\mathcal T)\leq \rho_{r,n}(d)
\end{equation*}
sur le rang de $\mathcal T$, où 
$$\rho_{r,n}(d)=\sum_{\sigma\geq 0} { r-1+\sigma \choose r-1 }\max\big(  d-(r+\sigma)(n-1)-1  , 0 \big)$$ désigne la {\bf constante de 
Castelnuovo-Harris}. D'après la généralisation de Harris \cite{Harris} d'un résultat classique de Castelnuovo sur les courbes,  $\rho_{r,n}(d)$ est égal au maximum des genre $g(V)$ pour $V$ une sous-variété algébrique irréductible de $\mathbb P^{n+r-1}$ de dimension $r$ et de degré $d$. 
\sk 

Un $d$-tissu $\mathcal T$ de type $(r,n)$ est {\bf de rang maximal} si ${\rm rg}(\mathcal T)= \rho_{r,n}(d)$.

\subsection{Le problème de l'algébrisation des tissus} 
Un résultat classique de Lie,  énoncé par lui  dans le cadre des surfaces de double translation, 
se traduit en géométrie des tissus par le fait qu'un  4-tissu en courbes dans le plan de rang maximal $\rho_{1,2}(4)=3$ est algébrisable.   
\sk 

Il est bien connu que pour envisager qu'un tel résultat se généralise aux tissus de type $(r,n)$,  il faut supposer $d$ ``suffisamment grand''  et, plus précisément, que $d$ soit supérieur ou égal à la {\bf borne de Poincaré}
$$
d_P(r,n)=(r+1)(n-1)+2.
$$

Un problème central en géométrie des tissus (explicitement formulé et commenté à  la toute fin de \cite{CGadd} par exemple\footnote{Voir aussi \cite{PTweb} et les références qui s'y trouvent.}) est celui 
de l'{\bf algébrisation des tissus}. Plus précisément, il s'agit de 
\begin{quote}
 {\it classifier les $d$-tissus de type $(r,n)$  qui sont de rang 
 maximal  $\rho_{r,n}(d)$
pour $r\geq 1, n\geq 2$ et 
$d\geq d_P(r,n)$.}  
\end{quote}
\sk 


Si ce problème a été abordé avec un certain succès dès les années 1930, ce n'est que très récemment qu'il a été résolu dans la plupart des cas. 
\begin{thm*}
{\it Un $d$-tissu de type $(r,n)$  
 de rang maximal $\rho_{r,n}(d)$ 
  est alg\'ebri\-sable au sens classique  si\sk
\begin{itemize}
\item[$\bullet$]  $r=1$ et $n=2$ avec  
$d=3$ ou $d=4$ 
 {\rm (}cf. \cite{Lie,Poincare}{\rm )}; \sk 
\item[$\bullet$]  $r=1$ et  $n\geq 3$  {\rm (}cf. \cite{Bol,Trepreau}{\rm )}; \sk 
\item[$\bullet$]  $r>1$ et   $n\geq 2$, avec  $d\neq (r+2)(n-1)+1
$ si $n>2$
     {\rm (}cf. \cite{PTweb}{\rm )}.\mk 
\end{itemize}}
\end{thm*}

D'autre part, d'apr\`es \cite{MPP}, on sait qu'il existe  des $d$-tissus plans  de rang maximal qui ne sont pas alg\'ebrisables pour tout $d\geq 5$. 
Restait  donc ouverte la question de   savoir si un tissu de type $(r,n)$  de rang maximal est alg\'ebrisable au sens classique  lorsque $r,n$ et $d$ vérifient
\begin{equation}
 \label{E:condiNUM}
r>1\,,  \qquad n\geq 3\, \qquad \mbox{ et } \,  \qquad d=d_{exc}(r,n), 
\end{equation}
où l'on a posé $$d_{exc}(r,n)=(r+2)(n-1)+1.$$

Vu les  résultats positifs que l'on vient de mentionner, on pourrait penser  que c'est encore le cas.  Or c'est l'inverse qui se produit: on montre dans cet article 
que la situation dans ce cas est plus riche que ce que l'on pouvait na\"ivement attendre, du moins pour certaines valeurs de  $n$.
\sk 

 Le cas des tissus de codimension 1 étant traité\footnote{Partiellement lorsque  $n=2$, mais c'est une autre histoire...}, on suppose   $r$ strictement plus grand que 1 dans toute la suite. Lorsque $r>1$ et $n\geq 2$ sont fixés, on note $d_P$ et $d_{exc}$ pour $d_P(r,n)$ et $d_{exc}(r,n)$ respectivement.
 
 \subsection{Quelques définitions et notations}
 Il convient  de poser quelques définitions afin de formuler plus clairement les choses.   \sk 
 
\subsubsection{}
 \'Etant donnés: \smallskip
\begin{itemize}
 \item[$\bullet$]    une variété projective irréductible $X$;\smallskip
\item[$\bullet$]    une sous-vari\'et\'e alg\'ebrique  r\'eduite $Z
\subset   X$  de dimension $r$;\smallskip
\item[$\bullet$]    une variété algébrique $\Sigma$  de dimension $rn$    paramétrant une famille de cycles irréductibles  $C_\sigma \subset X$, avec $\sigma \in \Sigma$, dont l'élément général intersecte proprement 
$Z$ en dimension 0, 
 \medskip 
\end{itemize}
nous dirons que le triplet $(X,\Sigma,Z)$ est {\bf admissible} si  la relation d'incidence entre les éléments de $\Sigma$ et les points de $Z$ induit un tissu de type $(r,n)$ sur   un ouvert de Zariski de la variété $\Sigma$.  Dans ce cas, on note ${\mathcal T}_{(X,\Sigma,Z)}$ ce tissu et l'on appelle {\bf tissus algébriques d'incidence}  les tissus obtenus via cette construction.   
\sk 

 Comme cela l'a été expliqué dans \cite{PTweb}, c'est une g\'en\'eralisation de la notion classique de tissu alg\'ebrique. 
En effet, si  $V\subset \mathbb P^{n+r-1}$ est une vari\'et\'e alg\'ebrique de  dimension $r$ dont l'intersection avec un $({n-1})$-plan g\'en\'erique
$ P
\in G_{n-1}(\mathbb P^{n+r-1})
$ est un 0-cycle r\'eduit  en position générale dans $P$, alors le tissu  $\mathcal T_V$ classiquement associ\'e  \`a $V$  peut aussi être défini comme le  tissu alg\'ebrique d'incidence   associé au triplet admissible 
\begin{equation}
\label{E:TripletClassique}
\big(\mathbb P^{n+r-1}, G_{n-1}(\mathbb P^{n+r-1}),V\big) .\medskip 
\end{equation}

\subsubsection{} Lorsque $(X,\Sigma,Z)$ est un triplet comme ci-dessus, on peut définir la {\bf trace} d'une $r$-forme  rationnelle $\omega$ sur $Z$ relativement à la famille de cycles paramétrée par $\Sigma$, que l'on note ${\rm Tr}_\Sigma(\omega)$. On montre que c'est une $r$-forme différentiel\-le rationnelle sur $\Sigma$. Comme dans le cas classique décrit en {\it \ref{SS:RAdeTV}}, sous l'hypothèse que le triplet considéré est  admissible, 
on vérifie qu'une relation de la forme ${\rm Tr}_\Sigma(\omega)= 0$ s'interprète de façon naturelle   comme une relation abélienne du tissu 
 associé à ${(X,\Sigma,Z)} $. 
\sk

\subsubsection{} Nous dirons qu'un $d$-tissu algébrique d'incidence    ${\mathcal T}=\mathcal T_{(X,\Sigma,Z)}$  de type $(r,n)$ est un {\bf tissu algébrique exceptionnel} si \smallskip 
\begin{itemize}
\item[$\bullet$] son  rang 
 est maximal, égal à $\rho_{r,n}(d)$; \smallskip
\item[$\bullet$] la trace ${\rm Tr}_\Sigma$ induit un isomorphisme entre un espace de $r$-formes 
 rationnelles sur $Z$ et l'espace 
des relations abéliennes de ${\mathcal T}$;\sk 
 \item[$\bullet$] le tissu $\mathcal T$  n'est pas algébrisable au sens classique, c'est-à-dire n'est pas  analytiquement  équivalent (même localement) à un tissu  algébrique d'incidence  associé à un triplet admissible du type (\ref{E:TripletClassique}). \smallskip
\end{itemize}
\sk 

\subsubsection{}
Le problème de l'algébrisation des tissus fait apparaître un lien profond avec une classe particulière de variétés projective qui a été considérée dans \cite{PTgeom} et que l'on va maintenant définir. 
\sk 

Une sous-variété réduite et irréductible $X$  d'un espace projectif est  dite {\bf $n$-recouverte par des courbes rationnelles normales (CRN) de degré $q\geq 1$} si par $n$ points généraux de $X$ passe une telle courbe incluse dans $X$. 
 Si $\dim(X)=r+1$, l'espace projectif qu'engendre une telle $X$ est de dimension au plus $\rho_{r,n}(d)-1$, où $d=q+r(n-1)+2$ comme on le voit en combinant  \cite[Théorème 1.2]{PTgeom} et la preuve du Lemme 4.2 de \cite{PTweb}.  On définit alors $\boldsymbol{\mathcal X_{r+1,n}(q)}$ comme étant l'ensemble des $X$ $n$-recouvertes par des CRN de degré $q$ et qui vérifient 
 $\dim(  \langle X \rangle )=\rho_{r,n}(d)-1$. \sk 
 
 Soit $X$ une variété appartenant à $\mathcal X_{r+1,n}(q)$. 
Un $n$-uplet  $x=(x_1,\ldots,x_n)\in X^n$ est {\bf admissible} si les  $x_i$ sont des points lisses de $X$ et si les points de tout  $n$-uplet $x'=(x_i')_{i=1}^n$ suffisamment proche de $x$ (en particulier pour $x'=x$!) sont sur une même CRN de degré $q$ incluse dans $X$.  Une telle CRN  de  $X$ sera dite admissible. 
 Enfin, pour $m\leq n$, un $m$-uplet $(x_1,\ldots,x_m)\in X^m$ est admissible s'il peut être complété en un $n$-uplet admissible.  On notera $X_{\rm adm}$ le sous-ensemble ouvert de $X$ formé des points admissibles.   \sk 
   
   Pour $X\in \mathcal X_{r+1,n}(q)$, il existe une unique composante irréductible de la famille des CRN de degré $q$ incluses dans $X$ qui est $n$-recouvrante (cf. \cite[Lemme 2.8]{PTgeom}). Cette composante sera notée $\Sigma_q(X)$. 
   D'après \cite[Théorème 2.11]{PTgeom},  le  sous-ensemble $\Sigma_q(X)_{\rm adm}$ formé des courbes admissibles est un ouvert de $\Sigma_q(X)$  et est une variété complexe lisse de dimension $nr$.
    \mk 
   
Ces définition étant posées, on est en mesure de rappeler le résultat d'algébrisation général obtenu dans \cite{PTweb}.

  \subsection{Modèle canonique d'un tissu de rang maximal}\sk

 Soit  $\mathcal T=(\mathcal F_i)_{i=1}^d$ un germe de $d$-tissu   de type $(r,n)$ 
  de rang maximal avec $d\geq d_P$.  Pour simplifier, on note $\mathfrak A$ l'espace $\mathfrak A(\mathcal T)$ de ses relations abéliennes en $x_0$.  Par hypothèse, c'est un 
  espace vectoriel de dimension $\rho_{r,n}(d)>0$. 
 \mk 
 
\subsubsection{} 
\label{SS:141}
Soit $(M,x_0)$ le germe de variété de dimension $rn$ sur lequel $\mathcal T$ est défini. 
 Pour tout $x\in (M,x_0)$,  l'évaluation en $x$ des $j$-ièmes composantes des relations abéliennes de $\mathcal T$ définit une forme  linéaire sur $\mathfrak A$,   notée $ev_j(x)$,    à valeurs 
 dans l'espace vectoriel de dimension 1 des normales du tissu constant 
 induit  par $\mathcal T$ en $x$.   Ce dernier étant  
 de rang maximal, il est en particulier  \og de rang maximal en valuation 0\fg\footnote{Cela signifie que $\mathcal T$ admet
 exactement $d-r(n-1)-1$ relations abéliennes linéairement  indépendantes de valuation 0 en $x_0$; 
 c'est le nombre maximal possible, cf. \cite[Proposition 2.1]{PTweb}.} donc   $ev_j(x)$ n'est pas triviale et par conséquent son noyau ${\rm Ker}(ev_j(x))$ est  un hyperplan. Vu comme un point de $\mathbb P \mathfrak A^*$, on le note $\kappa_j(x)$.  
 
 En faisant varier $x$, on obtient un germe  d'application holomorphe
  $$\kappa_j: \big(M,x_0\big)\rightarrow \big( \mathbb P \mathfrak A^*, \kappa_j(x_0)\big).$$ 
 
Dans \cite{PTweb}, modulo le choix d'une base qui permet d'identifier $\mathfrak A$ avec son dual, on démontre les points suivants: \sk
\begin{enumerate}
\item[$(i).$] pour tout $j=1,\ldots,d$,  l'application $\kappa_{j}$ 
est une intégrale première de rang $r$ de   $\mathcal F_j$.  
L'image   $Z_j={\rm Im}(\kappa_j)$    est un germe en $ \kappa_j(x_0)$  de sous-variété analytique lisse de dimension $r$ de  $\mathbb P \mathfrak A^*$;

 \smallskip 
  \item[$(ii).$] pour tout $x\in (M,x_0)$, les $d$ points $\kappa_{1}(x),\ldots,\kappa_{d}(x)$  sont sur une même CRN de degré  $q=d-r(n-1)-2$
   de $\mathbb P{\mathfrak A}^*$, notée $\sigma(x)$;

 \sk
 \item[$(iii).$]  la réunion  des CRN $\sigma(x)$  quand $x$ varie dans $(M,x_0)$, est une sous-variété analytique lisse de  $\mathbb P{\mathfrak A}^*$,  de dimension $r+1$, dont l'adhérence de Zariski $X_{\mathcal T}$ appartient à  
  $\mathcal X_{r+1,n}(q)$; 
    \sk
 \item[$(iv).$]   pour tout $x\in (M,x_0)$, la courbe $\sigma(x)$ est admissible 
 et l'application 
 \begin{align}
 \sigma : (M,x_0)  & \longrightarrow \big( \Sigma_q(X_{\mathcal T})_{\rm adm}, \sigma(x_0)\big)
  \nonumber\\
  x & \longmapsto \sigma(x) \nonumber
 \end{align}
 est un germe de biholomorphisme;

 \sk  
  \item[$(v).$]   pour tout $j=1,\ldots,d$ et tout $x\in (M,x_0)$, le germe d'hypersurface $Z_j$ de $X_{\mathcal T}$ et la CRN  $\sigma(x)$ s'intersectent transversalement en  $\kappa_{j}(x)$;

  \sk 
    \item[$(vi).$]     les $Z_j$ sont les germes, en les points $\kappa_j(x_0)$, d'une même
 hypersurface algébrique  $Z_{\mathcal T}$ de $X_{\mathcal T}$ telle que pour tout 
 point $x\in (M,x_0)$, on a 
 $Z_{\mathcal T}\cdot \sigma(x)=\kappa_1(x)+\cdots +\kappa_d(x)
 $ en tant que 0-cycle;
  
 \sk
 \item[$(vii).$] l'application $\omega\mapsto (\kappa_j^\star( \omega))_{j=1}^d$ induit un isomorphisme 
linéaire entre un certain espace vectoriel $\mathcal A(Z_{\mathcal T})$ de $r$-formes rationnelles sur 
 $Z_{\mathcal T}$  et l'espace $\mathfrak A$ des relations abéliennes de $\mathcal T$.  De plus, les restrictions à $
 Z_{\mathcal T} \cap (X_{\mathcal T})_{\rm adm}
 $ des éléments de $\mathcal A(Z_{\mathcal T})$ sont abéliennes.
 \smallskip 
 \end{enumerate}
 
 On indique les résultats de \cite{PTweb} qui permettent d'établir les points listés ci-dessus: 
 $(i),\, (ii)$ et $(v)$ sont donnés par le Lemme 3.7 de cet article, résultat clef qui donne aussi que l'application $\sigma$ du point $(iv)$ est injective.  Le reste de $(iv)$ est obtenu dans la Section 4.2.  Le point $(iii)$ découle de  la Proposition 3.11 et du   Lemme 4.2. Enfin, $(vi)$ et $(vii)$ sont donnés par le  Théorème 4.6. \mk 
 
\subsubsection{Modèle canonique et algébrisation} Soit alors $ \mathcal T_{\rm can}$,  le  {\bf modèle canonique}\,   de $\mathcal T$,  défini comme son poussé-en-avant  par $\sigma$. 
  Du point $(iv)$ ci-dessus, il vient que c'est un tissu de type $(r,n)$ isomorphe à $\mathcal T$, qui ne dépend que de sa classe d'isomorphie.  Des points $(v)$ et $(vi)$, on déduit ensuite que $\mathcal T_{\rm can}$ n'est rien d'autre que le germe en $\sigma(x_0)$ du tissu  algébrique d'incidence associé  au triplet   $(X_{\mathcal T}, \Sigma_q(X_{\mathcal T}),Z_{\mathcal T})$ qui est  admissible. 
 On a ainsi obtenu une description algébrique uniforme des (modèles canoniques des) tissus de type $(r,n)$ de rang maximal. \mk

On rappelle (cf. \cite[Définition 1.8]{PTgeom}) qu'une  variété $X\in \mathcal X_{r+1,n}(q)$ est dite {\bf standard} s'il existe une sous-variété $Y$  de $\mathbb P^{r+n-1}$,  de dimension $r+1$ et  de degré minimal $n-1$ telle que les paires d'incidences $(X
,\Sigma_q(X))$ et $(Y, \Sigma_{n-1}(Y))$ sont birationnellement isomorphes, i.e.  
s'il existe un isomorphisme birationnel entre $X $ et $ Y$  qui induit une équivalence de même nature  entre les familles de courbes rationnelles 
 $\Sigma_q(X)$~et~$\Sigma_{n-1}(Y)$.
\mk

\`A partir des résultats de  \cite{PTgeom} et de \cite{PTweb}
, on peut formuler le  
\bt
\label{T:1}
\begin{enumerate}
\item  Le tissu $\mathcal T$ est  algébrisable au sens classique si et seulement si sa variété de Blaschke   $X_{\mathcal T}$ est standard. 
\sk  
\item C'est toujours le cas, sauf peut être lorsque $n\geq 3$ et $q=2n-3$. 
\end{enumerate}
\et
%
\hspace{-0.4cm}(Le premier point  s'obtient en combinant \cite[Théorème 1.23]{PTweb} et \cite[Théorème 1.9]{PTgeom} tandis que {\it 2.} découle facilement de \cite[Théorème 1.7]{PTgeom}).
\sk 

 De la discussion qui précède le Théorème \ref{T:1}, il ressort 
que la possibilité de construire  des tissus algébriques exceptionnels  va dépendre  en premier lieu de l'existence de variétés des classes $\mathcal X_{r+1,n}(q)$ qui sont {\bf exceptionnelles}, c'est-à-dire qui ne sont pas standards. D'après le  Théorème 1.11 de \cite{PTgeom},  ce n'est éventuellement 
 possible que lorsque  
 $$n\geq 3\qquad \mbox{ et  } \qquad q=2n-3.$$  
 
 On suppose ces conditions numériques vérifiées dans ce qui suit.
%

\subsubsection{Variétés exceptionnelles 
 et tissus  associés}

La détermi\-na\-tion des variétés exceptionnelles des classes $\mathcal X_{r+1,n}(2n-3)$  est un joli probl\`eme de g\'eom\'etrie  projective  qui a \'et\'e  abord\'e en premier lieu à la fin de \cite{PTgeom}, où  sont décrites deux familles de variétés  de cette sorte de  dimension quelconque pour $n=3$ et $n=4$  ainsi que deux exemples  tridimensionnels    pour $n=5$ et $n=6$. \smallskip 

Dans \cite[\S 4.3]{piriorusso}, les auteurs  donnent une construction plus générale d'éléments 
 exceptionnels de la classe  $\mathcal X_{r+1,3}(3)$  à partir de certaines algèbres non-associatives dites \og de Jordan\fg.  Plus récemment dans \cite[\S3.3.2]{PR-XJC}, ils démontrent que lorsque $n=3$, toute variété exceptionnelle 
 est du  type \og Jordan\fg \, qui vient d'être évoqué.  En particulier, sont obtenues  des classifications complètes explicites des variétés exceptionnelles  lisses ou de petite dimension des classes $\mathcal X_{r+1,3}(3)$ (cf. \cite[Theorem 5.7]{piriorusso}). 

\begin{table}[h]
\label{Table}
 \centering
 \begin{tabular}{|c|c|l|}
 \hline
 {\bf Variété  exceptionnelle}    & {\bf Classe} 
 &\; {\bf Entiers} $r,n$   \\
 \hline    
\begin{tabular}{c}
${\rm Seg}(\mathbb P^1\times Q)\subset \mathbb P^{2r+3}$ avec  
 \\   $Q\subset \mathbb P^{r+1}$  hyperquadrique lisse
\end{tabular}
& $\mathcal{X}_{r+1,3}(3)$
& \,  $r\geq 2$, $n=3$\,   \\ \hline    
\begin{tabular}{c}
Image de $\mathbb P^{r+1}$ par $|3H-2Q|$ avec 
 $Q\subset \mathbb P^{r+1}$ \\  variété quadrique lisse de dimension $r-2$
\end{tabular}
& $\mathcal{X}_{r+1,4}(5)$
&\,  $r\geq 2$  ,  $n=4$  \,  \\ \hline    
\begin{tabular}{c}
Grassmannienne lagrangienne \\
$ LG_3(\mathbb C^6)\subset \mathbb P^{11}$
\end{tabular} & $\mathcal{X}_{6,3}(3)$
 &  \, $r=5$, $n=3$ \,
    \\ \hline 
   \begin{tabular}{c}
Grassmannienne  usuelle \\ 
$ G_3(\mathbb C^6)\subset \mathbb P^{19}$
\end{tabular} & $\mathcal{X}_{9,3}(3)$ &\, $r=8$, $n=3$\,
    \\ \hline 
 \begin{tabular}{c}
Grassmannienne orthogonale \\
$ OG_6(\mathbb C^{12})\subset \mathbb P^{31}$
\end{tabular} & $\mathcal{X}_{15,3}(3)$ &  \, $r=14$, $n=3$ \,
    \\    \hline 
 \begin{tabular}{c}
Variété  homogène sous $E_7$ \\
$E_7/P_7 \subset \mathbb P^{55}$ 
\end{tabular} & $\mathcal{X}_{27,3}(3) $ &\,    $r=26$, $n=3$\, 
    \\ \hline 
\begin{tabular}{c}
Variété  de Véronèse \\
$v_3(\mathbb P^3)\subset \mathbb P^{19}$
\end{tabular} & $\mathcal{X}_{3,6}(9)$ &  \, $r=2$, $n=6$ \, 
    \\ \hline 
\begin{tabular}{c}
Projection tangentielle  \\
 $\tau_x( v_3(\mathbb P^3))\subset \mathbb P^{15}$
\end{tabular} & $\mathcal{X}_{3,5}(7) $ &\,  $r=2$, $n=5$   \,
    \\ \hline 
\begin{tabular}{c}
Double projection tangentielle \\
  $\tau_{xy}( v_3(\mathbb P^3))\subset \mathbb P^{11}$
\end{tabular} & $\mathcal{X}_{3,4}(5)$ &  \, $r=2$, $n=4$  \,
    \\ \hline 
\end{tabular}\vspace{0.25cm} 
\caption{Variétés exceptionnelles considérées dans cet article.
 }\vspace{-0.4cm}
\end{table}


 C'est \`a partir  des variétés exceptionnelles  
de la table  
 ci-dessus  que l'on construit des tissus algébriques exceptionnels dans cet article. 
  Si $X$ désigne une  variété  de cette table,  on a  la variante suivante du célèbre théorème d'addition d'Abel  pour la paire d'incidence $(X,\Sigma_{q}(X))$ (cf. \cite{Griffiths} ou Section \ref{S:Abel} plus bas):  
\begin{prop}
Soit   $Z$ une  hypersurface réduite de $ X$.
\begin{enumerate}
\item  La trace 
 d'une 
forme rationnelle  sur $Z$ 
est une forme rationnelle sur $\Sigma_q(X)$;
\item  La trace  d'une 
$r$-forme finie\footnote{Il s'agit des formes différentielles rationnelles sur $Z$ qui admettent  un modèle globalement holomorphe sur au moins une désingularisation de $Z$, cf. Section \ref{S:Abel}.}  sur $Z$ 
est identiquement nulle.
\end{enumerate}
\end{prop}

Ce résultat permet de construire des relations abéliennes pour un tissu algé\-brique d'incidence  $\mathcal T_Z$ sur $\Sigma_q(X)$ défini à partir d'une   hypersurface $Z$  de $X$ (qu'on suppose lisse pour simplifier): du  point  2. de la proposition ci-dessus, on déduit 
une injection linéaire 
$ H^0(Z,K_Z) \hookrightarrow \mathfrak A(\mathcal T_Z ) $. \sk

C'est en utilisant ce fait que l'on  construit des tissus algébriques exceptionnels 
dans la Section \ref{S:Tissusalgebriquesexceptionnels}. 
 Si $X$ désigne l'une des variétés de la T{\small{ABLE}} 1, on trouve  une certaine  classe  explicite $E_X$  dans le groupe de Picard de $X$ telle que,  pour une courbe $C\in \Sigma_q(X)$ générale, on a: 
$$(C \cdot E_X)={(r+2)(n-1)+1}={d}_{exc}.$$

  On vérifie alors que  $Z\in  | E_X|$  générique est lisse et telle que $h^0(Z,K_Z)
 =\rho_{r,n}({d}_{exc})$. On en déduit que le tissu $\mathcal T_Z$ est exceptionnel. 
 \sk 
 
 Résumons tout ce qui vient d'être dit  au moyen de l'énoncé suivant qui est le résultat principal de cet article\!: 
 
\bt
\label{T:DescriptionUniforme}
Soit $X$ l'une des variétés de la {\rm T{\small{ABLE}} 1}. Il existe une classe  $E_X\in {\rm Pic}(X)$ telle que si $Z\in | E_X|$  est lisse (ce qui est génériquement vérifié), le triplet $(X,\Sigma_q(X),Z)$ est admissible et définit  un tissu  algébrique d'incidence  exceptionnel $\mathcal T_Z$ de type $(r,n)$ sur $\Sigma_q(X)$. 
\et

Nous construisons donc des tissus  algébriques exceptionnels  de type $(r,n)$ 
dans les cas suivants:  $n=3,4$ et $r\geq 2$;   $n=4,5,6$ et $r=2$.   

Cela montre que 
 la question de l'algébrisation des tissus de rang maximal dans le cas (\ref{E:condiNUM}) est plus riche que dans le cas général, et cela d'autant plus que les ingrédients utilisés pour construire ces exemples de tissus algébriques exceptionnels sont liés à un autre domaine des mathématiques (la théorie des algèbres non-associatives)  que personne  n'avait  suspecté auparavant  d'entretenir des liens avec la théorie  des tissus de rang maximal.  \medskip 

   Il découle du théorème d'Abel--Inverse que les relations abéliennes d'un tissu algébrique grassmannien  associé à une sous-variété  $V$ d'un espace projectif  correspondent aux formes différentielles abéliennes de $V$\footnote{Il s'agit des formes différentielles méromorphes de degré maximal $\omega$ sur $V\subset \mathbb P^N$ telles que le courant $[\omega]\wedge [V]$ est $\overline{\partial}$-fermé (voir \cite{Barlet} ou \cite{HenkinPassare} pour des précisions).}. Il est naturel de chercher à savoir si cela s'étend aux tissus algébriques d'incidence.  Plus précisément, soit  $\mathcal T_Z$  un tissu exceptionnel du  
   Théorème \ref{T:DescriptionUniforme}, associé à une hypersurface réduite $Z$ d'une variété $X$ de la T{\small{ABLE}} 1. 
   D'après \cite{PTweb}, on sait que les relations abéliennes de $\mathcal T_Z$ correspondent aux  
 annulations des traces par rapport à la famille de 1-cycles $\Sigma_q(X)$, de certaines formes rationnelles sur $Z$.  On se demande si ces formes ne sont pas en fait abéliennes. 
 
   Si nous ne savons pas  répondre à cette question dans le cas général,  certains r\'esultats de \cite{piriorusso} permettent d'y apporter 
la réponse  suivante dans le  cas  particulier mais intéressant $n=3$: 

\bt
Soit  $\mathcal T$ un $d_{exc}$-tissu  exceptionnel de type $(r,3)$. Si  
  $X_{\mathcal T}$ est lisse, c'est l'un des exemples de  la T{\small{ABLE}} 1, l'hypersurface $Z_{\mathcal T}$ appartient au système linéaire   $\lvert  E_{X_{\mathcal T}}\lvert $ 
   et   la trace par rapport à la famille de courbes cubiques $\Sigma_3(X_{\mathcal T})$    
   induit un isomorphisme linéaire entre l'espace des formes abéliennes de $Z_{\mathcal T}$ 
   et l'espace des relations abéliennes du modèle canonique $\mathcal T_{can}=\mathcal T_{Z_{\mathcal T}}$ de $\mathcal T$. 
\et
\sk

Pour finir cette introduction, indiquons comment est organisée la suite de l'article.  Dans la Section \ref{S:Xexceptionnelles}, nous décrivons les variétés exceptionnelles des  classes $\mathcal X_{r+1,n}(2n-3)$ qui sont connues. On décrit en détail les variétés  de ce type à partir desquelles nous construisons des tissus algébriques exceptionnels.  La courte section \ref{S:Abel} concerne des résultats du même type que le  Théorème d'addition  d'Abel 
pour les paires d'incidence $(X,\Sigma_{2n-3}(X))$ quand $X$ est une variété exceptionnelle.  Ils serviront à  construire des relations abéliennes des tissus considérés dans la Section \ref{S:Tissusalgebriquesexceptionnels}.  Cette dernière 
 est la section la plus importante de l'article. 
C'est précisément là que sont construits en détail les tissus algébriques exceptionnels évoqués plus haut.  Pour finir, dans la Section \ref{S:Final}, on commente les résultats obtenus et on indique ce qu'il reste à faire pour arriver à une classification complète des tissus algébriques exceptionnels. 
\bk

\hspace{-0.45cm}{\bf Remerciement.} L'auteur remercie chaleureusement Francesco Russo pour son aide sur plusieurs points de géométrie algébrique.


 \section{Variétés exceptionnelles}
\label{S:Xexceptionnelles}
Dans cette section, nous faisons le point sur les variétés non-standards des classes $\mathcal X_{r+1,n}(2n-3)$ qui sont connues. En particulier, nous décrivons en détail les variétés exceptionnelles 
à partir desquelles nous construirons des tissus exceptionnels dans la Section \ref{S:Tissusalgebriquesexceptionnels}. 

Les exemples   ci-dessous ont été 
considérés pour la première fois dans   \cite{PTgeom} et \cite{piriorusso}, articles auxquels le lecteur pourra se référer.
Nous les présentons ici  par souci de complétude.

\subsection{Algèbres de Jordan de rang 3  et  variétés exceptionnelles des classes $\mathcal{X}_{r+1,3}(3)$}
\label{S:jordancubicXJ}
Une construction de variétés exceptionnelles dans les classes   ${\mathcal X}_{r+1,3}(3)$  a été présentée dans \cite[\S 4.3]{piriorusso}.  Elle repose sur la notion d'{\bf algèbre de Jordan}. Par défini\-tion, une telle  algèbre est une algèbre complexe commutative et unitaire  $J$ (toujours supposée de dimension finie dans cet article) qui n'est pas associative mais  dont le produit vérifie l'\og {\it identité de Jordan}\fg
$$x^2(yx)=\big(x^2y\big)x.$$ 

Par exemple, une algèbre associative commutative et unitaire est  une algèbre de Jordan. Plus généralement, si $A$ est une algèbre associative unitaire mais non commutative,    le produit  symétrisé 
$$x\cdot y=\frac{1}{2}\big(xy+yx\big)$$
 induit sur  $A$ une structure d'algèbre de Jordan, notée $A^+$.\smallskip

 On peut montrer qu'une algèbre de Jordan $J$ est une algèbre à  puissance associative. 
 Par définition,  le {\bf rang}    
   d'un élément $x\in J$ est la dimension, comme sous-espace vectoriel complexe, de la sous-algèbre 
      engendrée par $x$. 
      On définit alors le {\bf rang de $J$} comme le maximum des rangs 
   de ses éléments.    C'est un entier fini,  plus petit que  $\dim(J)$. 
  \smallskip

 On suppose dorénavant que $J$ est une algèbre de Jordan de rang 3.  On note $e$ son unité.  Un résultat classique  du domaine (cf. \cite[Chap. VI]{jacobson} par exemple) nous assure qu'il existe des formes polynomiales homogènes $T,S$ et $N$ 
  de degrés respectifs 1, 2 et 3  telles que pour tout   $x\in J$, on a
 $$
 x^3-T(x)x^2+S(x)x -N(x)e=0. 
 $$
 Par définition, la forme linéaire $T$ et la forme cubique $N$ sont respectivement la {\bf trace générique} et  la {\bf norme générique} de  $J$. On définit alors l'{\bf adjoint} $x^\#$ d'un élément $x$ de  $J$  
en posant  
 $$x^\#=x^2-T(x)x+S(x)e.$$
 Pour tout $x\in J$, on a   $xx^\#=x^\#x=N(x)e$ ainsi que $(x^\#)^\#=N(x)x$.
\sk 

Dans ce qui suit, il sera pratique de noter $r+1$ la dimension de $J$ et de poser 
 $$Z_2(J)=\mathbb C\oplus J\oplus J \oplus \mathbb C.$$
 
Par définition, la {\bf courbe cubique sur $J$}, notée $X_J$
, est l'adhérence de Zariski de l'image de l'application affine
\begin{align}
\label{nu}
\nu=\nu_{J}: J & \longrightarrow \; \;  \mathbb PZ_2(J) \simeq \mathbb P^{2r+3} \\
x  & \longmapsto    \big[1: x: x^\#: N(x)   \big].\nonumber 
\end{align}

La réciproque de cette application peut être considérée comme une carte sur $X_J$. Celle dernière  en admet une autre,  tout aussi naturelle: la réciproque de 
\begin{align}
\label{nuhat}
\hat \nu=\hat {\nu}_{J}: J & \longrightarrow
\; \;  \mathbb PZ_2(J) \simeq \mathbb P^{2r+3}  \\
y  & \longmapsto    \big[N(y):y^\#: y:1   \big].   \nonumber 
\end{align} 

 La réunion $X_J^0={\rm Im}(\nu_{J}) \cup\,  {\rm Im}(\hat {\nu}_{J})$ est un ouvert dense et lisse de $X_J$. 
 L'adhéren\-ce de Zariski de  $\nu(\mathbb C e)$ est une courbe cubique gauche incluse dans $X_J^0$, notée $C_J$. Celle-ci passe par les points $0_J=\nu_J(0)$, $e_J=\nu_J(e)$ et $\infty_J=[0:0:0:1]$ de $X_J$.\sk 
 
 On sait d'autre  part que le groupe conforme ${\rm Conf}(J)$ de $J$ \footnote{
 Il s'agit du
  sous-groupe des transformations birationnelles de $J$ engendré par les translations $x\mapsto x+x_0$, pour  $x_0\in J$, l'inversion $x\mapsto x^{-1}$
et les  transformations linéaires $g\in GL(J)$ telles qu'il existe $g^\#\in GL(J)$ vérifiant $g(x)^\#=g^\#(x^\#)$ pour tout $x\in J$.} s'identifie à un sous-groupe  du groupe  des automorphismes projectifs de $X_J\subset \mathbb PZ_2(J)$ et qu'il agit transitivement sur les triplets génériques de points distincts  de $X_J$.  En considérant les images de $C_J$ par les éléments de ${\rm Conf}(J)$ on montre  que par trois  points généraux de $X_J$ passe une cubique gauche incluse dans $X_J$, 
cf.  \cite[Proposition 4.7]{piriorusso}.  En d'autres termes, on obtient que 
  $X_J$ est un élément de la classe $ \mathcal {X}_{r+1,3}(3)$. \medskip

L'orbite  de $0_J$ sous l'action de ${\rm Conf}(J)$ contient $X_J^0$, elle est donc ouverte  et  dense. De ce fait, on peut  voir $0_J$ comme un point générique  de $X_J$. En particulier, ce point est admissible au sens de \cite{PTgeom}.  Dans la carte affine associée à l'application $\nu_J$, la projection tangentielle $\tau=\tau_{0_J}: X_J\dashrightarrow \mathbb P^{r+1}$ de centre l'espace projectif tangent  
 à $X_J$ en $0_J$  se lit
 $$
 x \longmapsto \big[  x^\# : N(x) \big].
 $$

  Pour   $x\in J$ fixé,  l'application $ \mathbb P^1\rightarrow X_J $, $ [\lambda:\gamma]\mapsto 
 \nu_J(\lambda\, e+\gamma\,  x)$  est un paramétrage projectif  d'une courbe rationnelle $C_{x}$ incluse dans $ {X_J}$.  Si $N(x)\neq 0$, on vérifie que c'est une cubique gauche élément de  $\Sigma_3({X_J})$.    
 
 D'autre part,   pour tout $t\in \mathbb C$, on a 
 \begin{align*}
 \tau\big(\nu_J(e+tx)\big)
 =(e,1)+ \, t\big(T(x)e-x,T(x)\big) +  \, t^2\big(x^\#,S(x)   \big)+t^3\big(0,N(x)\big) .  
 \end{align*}
 
On en déduit immédiatement que $\tau(C_x)$ est une cubique gauche dans $\mathbb P^{r+1}$ dès que $x$ est de rang 3.  Cela démontre que 
  \begin{equation}
  \label{E:tau(C)} 
\begin{tabular}{l} {\it l'image  par une projection tangentielle générale  d'une courbe} \\ {\it générale de la 
 famille $\Sigma_3({X_J})$  est une cubique gauche. }
\end{tabular}
 \end{equation}

Considérons maintenant le cas d'une variété  standard $S\in \mathcal X_{r+1,3}(3)$.  
\`A équivalence projective près, $S$ est l'un des deux scrolls $S_{1\ldots122}$ ou $S_{1\ldots13}$, cf.  \cite[Théorème 5.3]{PTgeom}. Par conséquent,   pour  tout $s\in S$, l'espace projectif tangent 
$S_s(1)$  à $S$ en $s$  intersecte  $S$ le long d'un sous-espace projectif 
  de dimension $r$  de $\mathbb P^{2r+1}$, qu'on note $\Pi_s$.  On peut  alors 
 vérifier qu'une  cubique gauche générale  incluse dans $S$ intersecte  $\Pi_s$. Cela implique en particulier qu'une telle courbe  rencontre le centre de la projection tangentielle $S\dashrightarrow \mathbb P^{r+1}$ de centre $S_s(1)$. On en   déduit que   
    \begin{equation}
     \label{E:tau(C)S}
\begin{tabular}{l} {\it l'image  par une projection tangentielle générale  d'une courbe} \\ {\it  générale de la 
 famille  $ \Sigma_3(S)$  est une conique. }
\end{tabular}
 \end{equation}
 
En comparant (\ref{E:tau(C)})  et (\ref{E:tau(C)S}), on obtient immédiatement le résultat suivant: 
\begin{prop}
\label{P:XJnonStandard}
  Soit $J$ une algèbre de Jordan de rang 3. La 
  variété  $X_J$ est une variété exceptionnelle  de la classe $\mathcal {X}_{r+1,3}(3) $.
  \end{prop}
  
   Au moyen des algèbres de Jordan de rang 3, on peut donc construire une vaste classe de variétés exceptionnelles dans le cas $n=3$. En fait, il découle  de \cite[Theorem 3.7]{PR-XJC} qu'il n'y a pas d'autres variétés    exceptionnelles en plus de celles-ci  dans les classes $\mathcal {X}_{r+1,3}(3) $. Nous n'utiliserons pas ce résultat dans la suite mais nous intéresserons plus particulièrement aux variétés exceptionnelles  $X_J$ lisses.  D'après \cite[Theorem 5.7]{piriorusso},  ce sont exactement les courbes cubiques  $X_J$ sur les algèbres de Jordan semi-simples de rang 3, 
    dont nous allons  rappeler la classification.


  \subsection{Algèbres de Jordan semi-simples de rang 3   et  variétés exceptionnelles lisses des classes $\mathcal X_{r+1,3}(3)$} 
  \label{S:jordancubicXJlisses}
  Tout d'abord, rappelons qu'une  algèbre de Jordan est dite {\bf simple} (resp. {\bf semi-simple}) si elle ne possède pas d'idéal  (resp. pas d'idéal résoluble)  non-trivial. D'après un résultat classique de la théorie  dû à Albert (voir par exemple \cite[Chap.\,V]{jacobson}), 
 une algèbre de Jordan $J$ est semi-simple si et seulement si elle est isomorphe à un produit direct d'algèbres de Jordan simples. 
Les algèbres de Jordan simples complexes étants classifiées  ({cf.} \cite[Chap.\,V \S7]{jacobson}), on obtient  facilement la classifications des algèbres de Jordan $J$ semi-simples de rang 3. 

\subsubsection{\bf Les quatre cas simples} 
  \label{S:Les quatres cas simples}
Soit $\mathbb {R}, \mathbb{C}, \mathbb{H}$ et $\mathbb{O}$, 
les quatre algèbres ``de Hurwitz'':  $\mathbb {R}$ et $ \mathbb{C}$ sont les corps des nombres réels et   complexes respectivement, $\mathbb H$ est l'algèbre des quaternions de Hamilton et $\mathbb O$ celle des octaves de Graves et Cayley.  

Si $
\mathbb{A}$ désigne l'une  de ces quatre algèbres, on note 
${\boldsymbol{{A}}}=\mathbb A\otimes_{\mathbb R} \mathbb C$ 
sa complexification.  On obtient ainsi une algèbre complexe munie d'une involution $\mathbb C$-linéaire $x\mapsto \overline{x}$ qui est un anti-isomorphisme ({\it i.e.}  $\overline{xy}=\overline{y}\, \overline{x}$ pour tout $x,y\in \boldsymbol{A}$) et  telle que $\lvert \lvert x \lvert \lvert^2=x\overline{x}$ soit scalaire ({\it i.e.} appartienne à la droite vectorielle $\mathbb C  \boldsymbol{1}$) quel que soit $x\in \boldsymbol{A}$.  De plus, $ \lvert \lvert \cdot  \lvert \lvert^2$ est une forme quadratique complexe non-dégénérée 
 sur ${\boldsymbol{{A}}}$. 
\smallskip

En tant qu'algèbres complexes, $\boldsymbol{R}, \boldsymbol{C}, \boldsymbol{H}$ et $\boldsymbol{O}$ sont respectivement de dimension $1,2,4$ et $8$  et il y a  des isomorphismes classiques
\begin{equation}
\label{E:A-Isom-Classique}
\boldsymbol{R}\simeq \mathbb C\, ,\qquad  \boldsymbol{C}\simeq \mathbb C\times \mathbb C\qquad \mbox{et}\qquad  \boldsymbol{H}\simeq M_2(\mathbb C).
\end{equation}

Pour $\boldsymbol{A}$ comme ci-dessus, on note ${\rm Herm}_3(\boldsymbol{A})$ l'ensemble des matrices carrées $3\times3$ à coefficients dans $\boldsymbol{A}$, qui sont hermitiennes: 
$$
{\rm Herm}_3(\boldsymbol{A})=\left\{ 
\begin{pmatrix}
\alpha_1 & a_1 & a_2\\
\overline{a_1} & \alpha_2 &  a_3\\
\overline{a_2}& \overline{a_3}   &   \alpha_3
\end{pmatrix} \; \Bigg\lvert  
\begin{tabular}{l} 
$\alpha_1,\alpha_2,\alpha_3\in \mathbb C$\vspace{0.1cm}\\ 
$a_1,a_2,a_3\in \boldsymbol{A}$
\end{tabular}
\right\}. 
$$

On vérifie que le produit $ M\cdot N=\frac{1}{2}\big( MN+NM \big)$
est bien à valeurs dans ${\rm Herm}_3(\boldsymbol{A})$ et on montre qu'il fait de 
${\rm Herm}_3(\boldsymbol{A})$ 
une algèbre de Jordan\footnote{Ce fait est facile à établir lorsque $\boldsymbol{A}$ est associative, mais est non-trivial pour   $\boldsymbol{O}$, {cf.} \cite[p.\,21]{jacobson}.} 
 de rang 3 qui est simple. De plus, on peut montrer que, à isomorphismes près, 
les algèbres ${\rm Herm}_3(\boldsymbol{A})$, 
$\boldsymbol{A}= \boldsymbol{R} , \boldsymbol{C} ,\boldsymbol{H}, \boldsymbol{O}  $ forment la liste complète des algèbres de Jordan simples et de rang 3, cf. 
\cite[p. 210]{jacobson}.
\sk 

On rappelle que    
    (\ref{E:A-Isom-Classique}) induit les isomorphismes 
suivants d'algèbres de Jordan 
${\rm Herm}_3(\boldsymbol{R})\simeq {\rm Sym}_3(\mathbb C)^{+}$, $
{\rm Herm}_3(\boldsymbol{C})\simeq  {\rm M}_3(\mathbb C)^{+}$ et 
${\rm Herm}_3(\boldsymbol{H})\simeq {\rm Alt}_6(\mathbb C)^{+}$, 
où ${\rm M}_3(\mathbb C)$ est l'algèbre des matrices complexes $3\times 3$,  $ {\rm Sym}_3(\mathbb C)$ sa sous-algèbre formée des matrices symétriques, tandis que ${\rm Alt}_6(\mathbb C)$ désigne l'ensemble des matrices complexes $6\times 6$ antisymétriques.
\sk

Dans les quatre paragraphes qui suivent, nous  décrivons aussi explicitement que possible les courbes cubiques sur les algèbres de Jordan simples ${\rm Herm}_3(\boldsymbol{A})$ pour  
$\boldsymbol{A}=  \boldsymbol{R} , \boldsymbol{C} ,\boldsymbol{H}, \boldsymbol{O}  $. Pour des preuves ou  davantage  de précisions, nous renvoyons le lecteur aux articles \cite{Freudenthal} et \cite{LandsbergManivel}. 

\paragraph{La courbe cubique sur  ${\rm Herm}_3(\boldsymbol{R})$} 
On munit l'espace vectoriel $\mathbb C^6$ de la forme symplectique standard 
$\omega= \sum_{i=1}^3dx_i\wedge dx_{i+3}$. Un sous-espace vectoriel $V\subset \mathbb C^6$ de dimension 3 est {\bf lagrangien} si la restriction de $\omega$ à celui-ci est indentiquement nulle.  L'ensemble de ces sous-espaces forme la {\bf variété grassmannienne lagrangienne}
$LG_3(\mathbb C^6)$. C'est une sous-variété de dimension 6 de 
  $G_3(\mathbb C^6)$, qui  est homogène sous l'action du  groupe symplectique complexe ${Sp}_6(\mathbb C)$ et donc est lisse.  Par restriction, le plongement de Pl\"ucker
de $G_{3}(\mathbb C^6)$ dans $\mathbb P^{19}$ induit un plongement de  $LG_3(\mathbb C^6)$ dans $\mathbb P^{13}$.  En considérant dorénavant 
$LG_3(\mathbb C^6)$ comme une sous-variété de $\mathbb P^{13}$ et modulo équivalence projective, on a: 
$$
X_{{\rm Herm}_3(\boldsymbol{R})}=LG_3(\mathbb C^6)\subset \mathbb P^{13}.
$$

 \paragraph{La courbe cubique sur ${\rm Herm}_3(\boldsymbol{C})$}
 Dans ce cas, la courbe cubique sur ${\rm Herm}_3(\boldsymbol{C})$ est la variété grassmannienne $G_{3}(\mathbb C^6)$ des sous-espaces de dimension 3 de $\mathbb C^6$, vue comme une sous-variété de $\mathbb P^{19}$ au moyen du plongement de Pl\"ucker: 
 $$
X_{{\rm Herm}_3(\boldsymbol{C})}=G_3(\mathbb C^6)\subset \mathbb P^{19}. 
 $$
\paragraph{La courbe cubique sur ${\rm Herm}_3(\boldsymbol{H})$}
Soit $\theta=\sum_{i=1}^{12}(dx_i)^2$ la forme euclidienne complexe standard sur $\mathbb C^{12}$.  Un sous-espace  de $\mathbb C^{12}$ est {\bf isotrope} si la restriction de $\theta$ à celui-ci est identiquement nulle. Par définition, la {\bf grassmanienne orthogonale} $OG_6(\mathbb C^{12})$ est l'ensemble des sous-espaces isotropes de dimension 6 de $\mathbb C^{12}$.  C'est une sous-variété de dimension 15 de $G_{6}(\mathbb C^{12})$ qui est homogène sous l'action de ${SO}_{12}(\mathbb C)$ et donc est lisse. Par restriction du  plongement de Pl\"ucker $
G_{6}(\mathbb C^{12})\hookrightarrow \mathbb P(\wedge^6 \mathbb C^{12})$, on obtient  un plongement de $OG_6(\mathbb C^{12})$ dans $\mathbb P^{31}$ et on a: 
 $$
X_{{\rm Herm}_3(\boldsymbol{H})}=OG_{6}(\mathbb C^{12})\subset \mathbb P^{31}. 
 $$

\paragraph{La courbe cubique sur ${\rm Herm}_3(\boldsymbol{O})$}
Dans \cite{Freudenthal}, 
Freudenthal a démontré  que $X_{{\rm Herm}_3(\boldsymbol{H})}\subset \mathbb P^{55}$ est homogène sous l'action du groupe conforme 
  de ${\rm Herm}_3(\boldsymbol{O})$ qui, dans ce cas, se révèle être 
 le groupe de Lie exceptionnel $E_7$.  Le sous-groupe $P_7$ qui stabilise un point est un certain sous-groupe parabolique maximal de $ E_7$.  
 On peut donc écrire
 $$
X_{{\rm Herm}_3(\boldsymbol{O})}=E_7/P_7\subset \mathbb P^{55}. 
 $$
\subsubsection{\bf Le cas semi-simple non simple} Soit $r$ un entier plus grand ou égal à 2. Si $q\in{\rm Sym}^2(V^*)$ est une forme quadratique sur 
 un espace vectoriel $V$ de dimension $r-1$, on définit explicitement  un produit bilinéaire symétrique sur la somme directe $\mathbb C\oplus V$ en posant 
  $$
  (\lambda,v)\cdot \big(\lambda',v'\big)=\big(\lambda\lambda'-q(v,v'),\lambda \, v'+\lambda'v\big) 
  $$
   pour $\lambda,\lambda'\in \mathbb C$ et $v,v'\in V$, avec  $q(v,v')=\frac{1}{2}(q(v+v')-q(v)-q(v'))$.    
   
   On vérifie que ce produit définit une structure d'algèbre de Jordan de rang 2 sur $
 \mathbb C\oplus V$, d'unité $e=(1,0)$  et dont la classe d'isomorphisme est donnée (en plus de la dimension) par le rang de la forme quadratique $q$.  On la notera  $\mathcal J^V_q$ ou plus simplement   $\mathcal J^{r}_q$ si $V=\mathbb C^{r-1}$. \sk 
 
 On vérifie que $\mathcal J^{r}_q$ est semi-simple (en fait simple si $r>3$) si et seulement si $q$ est non-dégénérée. Dans ce cas, le  produit direct $J^{r+1}_q=\mathbb C\times \mathcal J^r_q$ est   semi-simple de rang 3. L'adjoint et la norme d'un élement $x=(\alpha,(\lambda,v)) \in \mathbb C\times \mathcal J_q^r$ sont donnés par les formules suivantes: 
 $$
 x^\#=\big(\lambda^2+q(v), (\alpha\lambda, -\alpha v)\big)\qquad \mbox{et}\qquad N(x )=\alpha \big(\lambda^2+q(v)  \big).
 $$
 
 Il est alors facile d'expliciter la paramétrisation affine (\ref{nu}) et de vérifier que la courbe cubique sur $J^{r+1}_q$ est projectivement équivalente  au plongement de Segre du produit de $\mathbb P^1$ avec une hypersurface quadrique lisse $Q\subset \mathbb P^{r+1}$:  on a 
 $$
X_{J^{r+1}_q}=
{\rm Seg}\big( \mathbb P^1\times Q   \big)
\subset \mathbb P^{2r+3}. 
 $$
 
 On remarquera qu'il n'est pas nécessaire de supposer  $q$ non-dégénérée pour obtenir que $X_{J^{r}_q}$ est dans la classe $\mathcal X_{r+1,3}(3)$. Dans la Section 6.2 de \cite{PTgeom}, on a  montré  que 
 pour toute  hyperquadrique $Q$, le plongement de Segre 
 ${\rm Seg}( \mathbb P^1\times Q)$   est un élément non-standard de  la classe $\mathcal X_{r+1,3}(3)$. 
 
 Puisqu'on ne considère ici  que les variétés exceptionnelles  lisses et vu que 
la lissité de la courbe cubique sur l'algèbre  ${J^{r+1}_q}$ équivaut à la  semi-simplicité de celle-ci, on se restreint dans cet article au cas où la forme quadratique $q$  considérée  est non-dégénérée.


\subsection{Une famille de variétés exceptionnelles dans le cas $n=4$}
 \label{S:3H-2Q}
Nous présentons à nouveau  la famille de variétés non-standards lisses des classes $\mathcal {X}_{r+1,4}(5)$ décrite dans la Section 6.2 de  \cite{PTgeom}.\sk

On suppose ici que $r$ est un entier strictement plus grand que 1 et l'on désigne par $Q$ une variété  quadrique lisse de dimension $r-2$ dans $\mathbb P^{r+1}$.  On choisit un système de coordonnées homogènes $u_0,u_1,s_1,\ldots,s_{r}$ sur $\mathbb P^{r+1}$ tel que $Q$ soit découpée par le système d'équations
$u_0=u_{1}=0$ et $q(s)=0$,  où $q$ est une forme quadratique non-dégénérée en les $s_i$.   Soit $\Pi=\langle Q\rangle$ le sous-espace  projectif engendré par $Q$: c'est un sous-espace projectif de codimension 2 de $\mathbb P^{r+1}$.
\sk 

Soit $H$ la classe d'un hyperplan dans le groupe de Picard de $\mathbb P^{r+1}$.  
On note 
$  |   3H-2Q |  $ le système linéaire   formé des hypersurfaces cubiques de $\mathbb P^{r+1}$ avec  des points doubles le long de $Q$.  C'est le  projectifié du sous-espace vectoriel de $H^0(\mathbb P^{r+1}, \mathcal O_{\mathbb P^{r+1}}(3))$   
 engendré par  la famille 
\begin{equation}
\label{E:l3H-2Ql}
\Big\{ u_{\epsilon_1}u_{\epsilon_2} u_{\epsilon_3}\, , \, u_{\epsilon_1}u_{\epsilon_2}s_j
\, , \, u_{\epsilon_1}q(s)
 \Big\}_{
\substack{
\epsilon_i\in \{0,1\}  \\  j=1,\ldots,r
}
}\subset  H^0\big(\mathbb P^{r+1}, \mathcal O_{\mathbb P^{r+1}}(3)\big).
\end{equation}

On en déduit que   $  |   3H-2Q |  $ 
est  de dimension $3r+5$ et qu'il induit une application rationnelle 
$$\varphi=\varphi_{ |   3H-2Q|}: \mathbb P^{r+1}\dashrightarrow \mathbb P^{3r+5}, $$ 
dont l'image est une sous-variété non-dégénérée de $\mathbb P^{3r+5}$ de dimension $r+1$, notée $X_Q$.
\sk 

Si $p_1,\ldots,p_4$ sont quatre points génériques de $\mathbb P^{r+1}$, alors 
$P=\langle p_1,\ldots,p_4 \rangle$ est un 3-plan qui  intersecte $\Pi$ proprement le long d'une droite. Par généricité,   cette dernière  intersecte transversalement la quadrique $Q$ en deux points distincts $a_P$ et $b_P$ qui, avec les $p_i$ initiaux,  forment une collection de six points en position générale dans le 3-plan projectif $P$. Il existe donc une  courbe  cubique gauche $C_P\subset P$ qui passe par ces six points.  On vérifie que $\varphi(C_P)$ est une  quintique rationnelle. Les $p_i$ ayant été supposés génériques, cela montre, combiné avec le Corollaire 2.2 de \cite{PTgeom},  que $X_Q$ appartient bien à  la classe $ \mathcal X_{r+1,4}(5)$. \sk

D'après \cite[Théorème 5.3]{PTgeom},  pour une dimension $r+1$ fixée, il n'y a que deux variétés standards dans la classe 
$\mathcal X_{r+1,4}(5)$: les deux scrolls rationnels normaux  $S_{2\ldots24}$ et $S_{2\ldots233}$.  
D'autre part, on vérifie (cf. \cite[\!p. \!47]{PTgeom}) que 
l'intersection de $X_Q$ avec son espace projectif tangent    
 en l'un de ses points génériques  est une quadrique de dimension $r-1$. En particulier, cette intersection ne contient pas de sous-espace projectif de dimension $r$. Cela implique que  $X_Q$ n'est pas un scroll 
et donc qu'elle n'est pas standard. On a démontré la 
\begin{prop} Soit $r\geq 2$. La variété $X_Q\subset \mathbb P^{3r+5}$ image de $\mathbb P^{r+1}$ par l'application induite par le  système linéaire $|   3H-2Q |  $  
    est une variété exceptionnelle de la classe ${\mathcal X}_{r+1,4}(5)$.
\end{prop}


\subsection{La variété de Veronese $v_3(\mathbb P^3)$ et ses projections tangentielles}
Par six  points en position générale de $\mathbb P^3$ passe une cubique gauche. On en déduit que $v_3(\mathbb P^3)\subset \mathbb P^{19}$ appartient à la  classe $\mathcal X_{3,6}(9)$. De plus,  $v_3(\mathbb P^3)$  n'est pas standard  en tant qu'élément de cette classe. 
 En effet, il découle du Théorème 5.3 de \cite{PTgeom} que les deux scrolls $S_{455}$ et $S_{446}$ sont  exactement les exemples standards dans ce cas (à équivalence projective près) et il est clair que $v_3(\mathbb P^3)$ n'est pas un scroll.
 \sk

Donnons nous maintenant   trois points distincts  $x, y$ et $z$ de $v_3(\mathbb P^3)$, supposés ``non-alignés'', c'est-à-dire non situés sur l'image par  $v_3$ d'une droite de $\mathbb P^3$. On vérifie que le 3-uplet $(x,y,z)$ est admissible.  
Si l'on note  respectivement  $\tau_x$,  $\tau_{xy}$ et $\tau_{xyz}$ les restrictions à $v_3(\mathbb P^3)$ des projections tangentielles en $x$, $x$ et $y$ et  $x, y$ et $z$, on vérifie que
 les variétés 
$$\tau_x\big(v_3(\mathbb P^3)\big)
\subset \mathbb P^{15} \, ,
\,  \qquad 
  \tau_{xy}\big(v_3(\mathbb P^3)\big)\subset \mathbb P^{11}  \qquad \mbox{et}\qquad
 \tau_{xyz}\big(v_3(\mathbb P^3)\big)
 \subset \mathbb P^7
 $$
sont  des  exemples  de variétés non-standards des classes $\mathcal X_{3,5}(7)$, $\mathcal X_{3,4}(5)$  et $ \mathcal X_{3,3}(3)$  respectivement.  On notera cependant que les deux derniers exemples  ainsi obtenus  ne sont  pas nouveaux:   $ \tau_{xy}(v_3(\mathbb P^3))$ coïncide avec la variété $X_Q$ construite dans la section précédente quand  $r=2$, et la triple projection tangentielle 
$\tau_{xyz}(v_3(\mathbb P^3))$ est projectivement équivalente au plongement de Segre de    $\mathbb P^1\times \mathbb P^1\times \mathbb P^1$ dans $ \mathbb P^7$, qui n'est rien d'autre que la courbe cubique $X_{J_q^3}$.

 \section{Un théorème d'addition d'Abel pour les paires \\ d'incidence  $(X,\Sigma_q(X))$ lorsque $X\in \mathcal X_{r+1,n}(q)$}
\label{S:Abel}

C'est via le Théorème d'addition d'Abel qu'on obtient des relations abéliennes des tissus algébriques grassmanniens à partir des formes holomorphes (ou plus généralement des formes abéliennes) sur les variétés projectives associées.  Dans cette section, on spécialise des résultats de Griffiths 
\cite{Griffiths} en 
généralisant cela aux tissus algébriques d'incidence associés aux  hypersurfaces des variétés des classes $\mathcal X_{r+1,n}(q)$.

\subsection{Variantes de Griffiths du théorème d'Abel}
\label{S:Abel-Griffiths}

Soit  $V$ une variété projective de dimension pure  $n$ que l'on suppose réduite mais pas forcément lisse ou irréductible.

Soit $\Omega$ une $k$-forme rationnelle sur $V$. 
 Elle est dite {\bf finie}    
  si pour une (ou de façon équivalente, pour toute) désingularisation $\mu:V'\rightarrow V$, la tiré-en-arrière  par $\mu$ de la restriction de $\Omega$ à $V_{\rm reg}$
 est holomorphe et  se prolonge en une forme holomorphe sur $V'$ tout entier, cf. \cite{Griffiths,HenkinPassare}.   Lorsque $k=\dim(V)$, ces formes peuvent également être décrites comme les sections globales  de l'image directe par $\mu$ du faisceau canonique de $V'$.    
 On note  $\tilde \Omega^{n}_V=\mu_\star(K_{V'})$. 
 
 Il découle immédiatement de leur définition que l'image inverse par un morphisme  d'une  forme  finie est encore finie; que 
 dans le cas où $V$ est lisse, ces formes ne sont rien d'autre que les $k$-formes holomorphes globales sur $V$. 
 On en déduit immédiatement qu'une $k$-forme finie sur une variété unirationnelle est forcément identiquement nulle si $k>0$.

 \medskip
Rappelons les versions très générales du Théorème d'addition d'Abel données par Griffiths dans 
\cite{Griffiths}.  Soit $F:V\dashrightarrow W$ une application rationnelle dominante génériquement finie. 
 Pour  $w\in W$ générique, le germe de $F$ en $v$, noté $F_v$, est étale en tout point    $v$ de la fibre $F^{-1}(w)$.  
Celle-ci étant finie,   pour une $k$-forme rationnelle $\Omega$ comme ci-dessus, la somme 
$\sum_{v\in F^{-1}(w)} (F_v)_*(\Omega)$ définit un germe de forme méromorphe sur $W$ en $w$.  Tous les germes obtenus ainsi se recollent pour former une $k$-forme méromorphe, a priori seulement définie sur un ouvert de Zariski de $W$.  Cette forme, 
est appelée  la  {\bf trace de $\Omega$ par $F$}.
\bt[\cite{Griffiths}]
\label{T:AbelGriffiths}
\begin{enumerate}
\item  La trace de $\Omega$ par $F$ 
 se prolonge  sur $W$ tout entier en une forme rationnelle,  notée ${\rm Tr}_F(\Omega)$.
\item Si $\Omega$ est finie, alors ${\rm Tr}_F(\Omega)$   l'est également.  
En particulier,  celle-ci  est identiquement nulle  si $W$ est unirationnelle.
\end{enumerate}
\et


\subsection{Applications à la construction de relations abéliennes} 
Le but ici est d'expliquer comment le théorème d'Abel de Griffiths rappelé dans la section précédente  s'utilise pour construire des relations abéliennes des tissus algébrique d'incidence construits à partir d'une paire $(X,\Sigma_q(X))$ lorsque  $X$ est un élément d'une des classes $ \mathcal X_{r+1,n}(q)$. \mk

Tout d'abord, il nous semble intéressant d'expliciter la condition intuitivement naturelle sous laquelle il est possible d'associer un tissu de type $(r,n)$ à une hypersurface $Z$ de $X$. 
   \begin{prop}
\label{P:XVSigmaAdm}
Le triplet $(X,\Sigma_q(X),Z)$ est admissible si et seulement si $Z$ est réduite et rencontre l'ouvert 
 des points admissibles de $X$.
\end{prop}
\begin{proof}
Les conditions de l'énoncé sont clairement nécessaires, montrons qu'elles sont suffisantes. 
Sous l'hypothèse qu'elles sont vérifiées, 
on peut trouver une courbe admissible $\sigma\in\Sigma_q(X)$
non-incluse dans $Z$ qui intersecte la partie régulière de $Z$.   
D'autre part, 
d'après le Théorème 2.9 et le Théorème 2.11 de \cite{PTgeom}, 
on sait que $X$ est lisse le long de $\sigma $ 
et que  $N_{\sigma/X}=\mathcal O_{\sigma }(n-1)^{\oplus r}$.   Ce fibré étant très ample,  on peut déformer $\sigma$  dans $\Sigma_q(X)_{\rm adm}$ afin   d'éviter tout sous-ensemble algébrique $S\subset X$ de codimension au moins 2, voir \cite[Proposition II.3.7]{kollar}. 

 En appliquant cela à $S=Z_{\rm sing}$, on 
obtient  qu'en prenant  $\sigma\in \Sigma_q(X)_{\rm adm}$ générale,  le 0-cycle  $\sigma \cdot Z$ est supporté dans  la partie régulière de $Z$.  \'Etant admissible,  $\sigma$ est très libre et donc on peut la déformer en $\sigma '\in \Sigma_q(X)$,  elle aussi admissible,   qui intersecte $Z_{\rm reg}$ transversalement.  On en déduit alors facilement que  $(X,\Sigma_q(X),Z)$ est admissible. \end{proof}

On suppose dorénavant que $Z$ est une hypersurface réduite de $X$  qui rencontre $X_{\rm adm}$. Le triplet $(X,\Sigma_q(X),Z)$ est donc admissible et définit un tissu de type $(r,n)$ sur $\Sigma_q(X)$, noté 
$\mathcal T_Z$. 
Soit  $\sigma_0\in \Sigma_q(X)$, une courbe  admissible de $X$ qui intersecte $Z_{\rm reg}$ transversalement en $d$ points distincts $\kappa_1(\sigma_0),\ldots,\kappa_d(\sigma_0)$.  
 Soient   $\kappa_j: (\Sigma_q(X),\sigma_0)\rightarrow (Z,\kappa_j(\sigma_0))$, $j=1,\ldots,d$,  
 les submersions holomorphes 
 telles que $\sigma \cdot Z=
  \kappa_1(\sigma)+ \cdots+\kappa_d(\sigma)$ 
  quel que soit  $\sigma\in (\Sigma_q(X),\sigma_0)$.  
Les applications $\kappa_j$ sont des intégrales premières des feuilletages du germe  de $\mathcal T_Z$ en $\sigma_0$,   noté $\mathcal T_{Z,\sigma_0}$. \sk

 \bco
 \label{C:FinieAW}
La trace induit un morphisme linéaire injectif 
\begin{align}
\label{E:Tr}
H^0\big(Z,\tilde \Omega_Z^r\big) & \longrightarrow \mathfrak A\big(\mathcal T_{Z,\sigma_0}\big) \\
 \Omega & \longmapsto \big(   \kappa_j^\star(\Omega) \big)_{j=1}^d  \nonumber
\end{align}
et par conséquent, on a  $p_g(Z)\leq  {\rm rg}(\mathcal T_{Z,\sigma_0})$. 
 \eco
 \begin{proof}  
 On note $\overline{\Sigma}_q(X)$ l'adhérence de Zariski de $\Sigma_q(X)$ dans la variété  des courbes rationnelles de degré $q$ contenues dans $X$.  
 Soit $\mathcal I_{Z,\Sigma_q(X)}$ l'adhérence de Zariski dans $Z\times  \overline{\Sigma}_q(X)$ de l'ensemble des paires  $(z, \sigma) \in Z \times \overline{\Sigma}_q(X)$ telles que $z$ appartient au support de $\sigma$.  Par restriction des deux projections du produit $Z\times  \overline{\Sigma}_q(X)$ sur chacun de ses deux facteurs, on obtient un diagramme d'incidence
 \begin{equation*}
    \xymatrix@R=0.3cm@C=1.3cm{  
 &  \mathcal I_{Z,\Sigma_q(X)} 
 \ar@{->}[dl]_{\mu}  \ar@{->}[dr]^{\nu}
  &   \\   
Z  &   & \overline{\Sigma}_q(X).}
 \end{equation*}

Si $\Omega$ est une $r$-forme finie  sur $Z$, son tiré-en-arrière  $\mu^\star(\Omega)$ est une forme finie sur 
$ \mathcal I_{Z,\Sigma_q(X)} $.  D'autre part, l'application $\nu$ est génériquement finie (elle est génériquement $k:1$, où 
 $k$ désigne le degré d'un 0-cycle $Z\cdot \sigma$  pour 
$\sigma\in \Sigma_q(X)$ admissible générale).  Du  second point du Théorème  \ref{T:AbelGriffiths} il découle donc  que   ${\rm Tr}_\nu(\mu^\star(\Omega))$ est une forme finie sur $\overline{\Sigma}_q(X)$.  Comme $\Sigma_q(X)$ est unirationnelle,  ${\rm Tr}_\nu(\mu^\star(\Omega))$ est identiquement nulle d'après le Lemme 2.8 de \cite{PTgeom}.  Par ailleurs, on vérifie aisément que 
${\rm Tr}_\nu(\mu^\star(\Omega))$ n'est rien d'autre que la trace de $\Omega$ par rapport à la famille $\Sigma_q(X)$ et s'écrit aussi localement $\sum_{j=1}^d \kappa_j^\star(\Omega)$ avec les notations introduites plus haut. 
 \sk 
 
 Pour des raisons  de degré, on a  $d\Omega=0$ et donc 
$d\kappa_j^\star (\Omega)=\kappa_j^\star(d\Omega)=0$ pour tout $j=1,\ldots,d$. Du fait que les $\kappa_j$ sont des intégrales premières de $\mathcal T_{Z,\sigma_0}$, il vient que le $d$-uplet  $(\kappa_j^*(\Omega))_{j=1}^d$ peut être vu comme une relation abélienne de ce  tissu. 
 On en déduit que l'application (\ref{E:Tr}) est bien à valeurs dans 
 l'espace des relations abéliennes de $\mathcal T_{Z,\sigma_0}$ 
 et donc est bien définie. 
 Puisqu'elle est  clairement  linéaire et injective, on obtient le corollaire.
\end{proof}

Nous finirons cette section par le  commentaire suivant: si le résultat ci-dessus est suffisant pour le but que nous nous sommes fixé dans cet article (à savoir, construire des tissus algébriques exceptionnels de rang maximal), la considéra\-tion du cas  des tissus algébriques grassmanniens nous fait nous demander s'il n'est pas possible d'améliorer le Corollaire \ref{C:FinieAW} et de prouver que l'application (\ref{E:Tr})  se prolonge à l'espace 
$H^0(Z,\omega^r_Z)$ 
des $r$-formes différentielles abéliennes\footnote{Nous appelons ``formes différentielles abéliennes" les 
sections globales sur $Z$ du faisceau $\omega^r_Z$ introduit par Barlet \cite{Barlet}   dans le cadre plus général  des variétés analytiques de dimension pure.} sur $Z$ et induit un isomorphisme entre celui-ci et $\mathfrak A(\mathcal T_{Z,\sigma_0})$, auquel cas on aurait ${\rm rg}(\mathcal T_{Z,\sigma_0})=h^0(Z,\omega^r_Z)$.  

Cela reviendrait essentiellement à démontrer des analogues du Théorème d'addition d'Abel et du Théorème d'Abel--Inverse,   non plus pour les paires  grassmanniennes   $(\mathbb P^N, G_k(\mathbb P^N))$, mais pour les paires d'incidences plus générales $(X,\Sigma_q(X))$ avec  $X\in  \mathcal X_{r+1,n}(q)$.


\section{Tissus algébriques exceptionnels}
\label{S:Tissusalgebriquesexceptionnels}
Dans cette section,  on construit des tissus algébriques exceptionnels à partir des variétés des classes $\mathcal X_{r+1,n}(2n-3)$ considérées dans la Section \ref{S:Xexceptionnelles}. 
 Plus précisément, si  $X$ désigne l'une de ces variétés, on trouve une certaine classe très ample $E_X$ dans le groupe de Picard de $X$ et on montre que $(X,\Sigma_{2n-3}(X),Z)$ est admissible pour $Z\in \lvert E_X\lvert$ générique.  Quand $Z$ est lisse, on utilise alors un résultat classique d'adjonction pour calculer $h^0(K_Z)$ et en déduire que le tissu associé à $(X,\Sigma_{2n-3}(X),Z)$ est de rang maximal et donc est exceptionnel.
\sk

La méthode de construction étant toujours la même, nous  la décrivons en détail dans le premier cas traité et sommes un peu  plus succincts
dans les cas qui suivent.

\subsection{Tissus  exceptionnels construits à partir de  $X_{{\rm Herm}_3(\boldsymbol{A})}$,  
avec  $\boldsymbol{A}= 
\boldsymbol{R}, \boldsymbol{C}, \boldsymbol{H}$ ou $ \boldsymbol{O}$}

Si $\boldsymbol{A}$ est l'une des quatre algèbres 
$\boldsymbol{R}, \boldsymbol{C}, \boldsymbol{H}$ ou $ \boldsymbol{O}$, 
on note pour simplifier $X_{\!\boldsymbol{A}}$ la courbe cubique sur l'algèbre de Jordan  ${\rm Herm}_3(\boldsymbol{A})$ décrite dans la Section \ref{S:jordancubicXJ}. 
\sk

Si  $G_{\!\boldsymbol{A}}$  désigne le groupe conforme de ${\rm Herm}_3(\boldsymbol{A})$,  on a vu que $ X_{\!\boldsymbol{A}}$ est homogène sous l'action de $G_{\!\boldsymbol{A}}$.  Cette variété étant   également  projective, on a  $X_{\!\boldsymbol{A}}=G_{\!\boldsymbol{A}}/P_{\!\boldsymbol{A}}$    pour un certain sous-groupe parabolique  $P_{\!\boldsymbol{A}} $ de $G_{\!\boldsymbol{A}}$ (cf. \cite[\!p.\,135]{Humphreys}),  qui
est maximal pour l'inclusion   parmi les sous-groupes paraboliques de $G_{\!\boldsymbol{A}}$. 
Comme le groupe de Picard de $G_{\!\boldsymbol{A}}/P_{\!\boldsymbol{A}}$ s'identifie au groupe des caractères de $P_{\!\boldsymbol{A}}$ d'après \cite[\!Theorem 4]{Popov}, il découle de la maximalité de 
$P_{\!\boldsymbol{A}}$ 
que ${\rm Pic}(X_{\!\boldsymbol{A}})$ est libre, sans torsion et  engendré par un seul   élément ample, noté $H_{\!\boldsymbol{A}}$. 
En fait, dans les cas considérés, $H_{\!\boldsymbol{A}}$ est la classe du fibré 
$\mathcal O_{\!X_{\!\boldsymbol{A}}}(1)$ 
associé au plongement 
$ X_{\!\boldsymbol{A}}\subset \mathbb P Z_2({\rm Herm}_3({\!\boldsymbol{A}}))$ décrit Section \ref{S:jordancubicXJ}.\sk

On sait aussi (cf. \cite[\!V.1.4]{kollar}) que $ X_{\!\boldsymbol{A}}$ est une variété de Fano, c'est-à-dire que  dans le groupe de Picard de  $X_{\!\boldsymbol{A}}$, on a
\begin{equation}
\label{E:KXA}
K_{ X_{\!\boldsymbol{A}}}= - i( X_{\!\boldsymbol{A}})\, H_{\!\boldsymbol{A}}, 
\end{equation}
 pour un certain entier  strictement positif $i( X_{\!\boldsymbol{A}})$,   appelé  l'{\bf indice} de $ X_{\!\boldsymbol{A}}$. \sk

     Soit $M_{0,1}(
X_{\!\boldsymbol{A}})$  la variété  des droites 1-pointées
  incluses dans $X_{\!\boldsymbol{A}}$\footnote{Plus rigoureusement, $M_{0,1}(
X_{\!\boldsymbol{A}})$  est la variété des classes d'isomorphismes de triplets $(\mathbb P^1,x,\mu)$ où $x$ est un point de $\mathbb P^1$ et $\mu: \mathbb P^1\rightarrow X_{\!\boldsymbol{A}}$ un paramétrage projectif d'une droite incluse dans $X_{\!\boldsymbol{A}}$, cf. \cite[\!\S 1.1]{fp}.}.
D'après \cite[\!Theorem 2]{fp}, si $D$ est une droite contenue dans $ X_{\!\boldsymbol{A}}$, on a $\dim (M_{0,1}(
X_{\!\boldsymbol{A}}))= \dim (X_{\!\boldsymbol{A}})+\deg (c_1(X_{\boldsymbol{A}})\lvert_D)-2$.  
La variété $X_{\!\boldsymbol{A}}$ étant homogène, on a 
$\dim( M_{0,1}(
X_{\!\boldsymbol{A}}))=\dim (X_{\!\boldsymbol{A}})+\dim( M_{\!\boldsymbol{A},o} )$ où 
$M_{\!\boldsymbol{A},o} $ désigne la famille des droites incluses dans $X_{\boldsymbol{A}}$ qui passent par un point fixé $o$ de $X_{\!\boldsymbol{A}}$. 
Par ailleurs, de  \eqref{E:KXA}, il vient  $\deg(c_1(X_{\boldsymbol{A}})\lvert_D)= 
i( X_{\!\boldsymbol{A}})\, \deg(H_{\!\boldsymbol{A}}\lvert_D)=i(X_{\!\boldsymbol{A}})$. On en déduit 
que $ i(X_{\!\boldsymbol{A}})=\dim( M_{\!\boldsymbol{A},o} ) +2$. 

 Il est connu  que les variétés  $M_{\!\boldsymbol{A},o}$ pour $
\boldsymbol{A}= 
\boldsymbol{R}, \boldsymbol{C}, \boldsymbol{H}$ ou $ \boldsymbol{O}$ sont les quatre variétés de Severi (voir \cite[Section 6]{piriorusso} ou \cite[Section 1.1]{LandsbergManivel}).  
Puisque  ces quatre variétés peuvent être vues 
 comme des complexifications des plans projectifs $\mathbb R\mathbb P^2, \mathbb C\mathbb P^2, \mathbb H\mathbb P^2$ et  $ \mathbb O\mathbb P^2$, on déduit une façon de  calculer  $ i(X_{\!\boldsymbol{A}})$ explicitement:   
%
quelle que soit  $\boldsymbol{A}$, on a 
 $$  i(X_{\!\boldsymbol{A}})=2\dim(\boldsymbol{A})+2.$$

 On récapitule et explicite dans le  tableau suivant certaines des assertions mention\-nées ci-dessus. 
\begin{table}[h]
 \centering
 \begin{tabular}{|c|c|c|c|c|c|c|}
 \hline
  $\boldsymbol{A}$   &  $\boldsymbol{{\rm Herm}_3(\boldsymbol{A})}$    &$\boldsymbol{G_{\!\boldsymbol{A}}}$     &  $\boldsymbol{X_{\!\boldsymbol{A}}}$  &   $\boldsymbol{M_{\!\boldsymbol{A},o}}$   & $\boldsymbol{i(X_{\!\boldsymbol{A}})}$ \\
 \hline   
 $\boldsymbol{R}$    &  ${\rm Sym}_3(\mathbb C)^+$ &     $Sp_6(\mathbb C)$   &   $LG_3(\mathbb C^6)$  &   $v_2(\mathbb P^2)$&4\\
 \hline
 $\boldsymbol{C}$    &  ${\rm M}_3(\mathbb C)^+$ &  $SL_6(\mathbb C)$   & $G_3(\mathbb C^6)$   &  ${\rm Seg}(\mathbb P^2\times \mathbb P^2)$    &6  \\
 \hline   
 $\boldsymbol{H}$    &  ${\rm Alt}_6(\mathbb C)^+$ & $SO_{12}(\mathbb C)$  & $OG_6(\mathbb C^{12})$  &  $G_2(\mathbb C^6)$ &10 \\
 \hline   
 $\boldsymbol{O}$    & $ {\rm Herm}_3(\boldsymbol{O})$ &$E_7$      & $E_7/P_7$  &   $\boldsymbol{O}\mathbb P^2$ &18 \\
 \hline   
\end{tabular}
\end{table}

Pour simplifier l'écriture, on pose 
 $\Sigma_{\boldsymbol{A}}=\Sigma_3(X_{\!\boldsymbol{A}})$ dans ce qui suit.\smallskip 

Pour toute   algèbre $\boldsymbol{A}$, on définit une classe très ample   
dans le groupe de Picard de  $X_{\!\boldsymbol{A}}$ 
en posant
$$
E_{\!\boldsymbol{A}}=\big(i(X_{\!\boldsymbol{A}})+1\big) H_{\!\boldsymbol{A}}
\, .
$$

La variété $X_{\!\boldsymbol{A}}$ étant  homogène, tous  ses points sont admissibles.  De ce fait,  si 
 $Z$ est un élément réduit du système linéaire  $| E_A | $, il rencontre bien 
$(X_{\!\boldsymbol{A}})_{\rm adm}=X_{\!\boldsymbol{A}}$ et par conséquent  le triplet $(X_{\!\boldsymbol{A}}, \Sigma_{\!\boldsymbol{A}}  ,Z )$ est admissible d'après la Proposition \ref{P:XVSigmaAdm}.  

On note $\mathcal T_Z$ le tissu algébrique d'incidence  défini par un tel  triplet.  
C'est un tissu sur $\Sigma_{\!\boldsymbol{A}}$ de  codimension égale à la dimension de $Z$, à savoir 
$$r_{\!\boldsymbol{A}}=\dim(X_{\!\boldsymbol{A}})-1=3\dim(\boldsymbol{A})+2.$$

 Comme $\dim(\Sigma_{\!\boldsymbol{A}})=3\, r_{\!\boldsymbol{A}}$,  le tissu 
$\mathcal T_Z$ est de  type $(r_{\!\boldsymbol{A}},3)$ et d'ordre  un certain entier $d_{\boldsymbol{A}}$ que l'on détermine facilement. 
En effet, 
vu qu'une courbe  admissible  $C\in \Sigma_{\boldsymbol{A}} $ est une cubique dans $X_{\!\boldsymbol{A}}\subset \mathbb PZ_2({\rm Herm}_3(\boldsymbol{A}))$ et puisque    
  ce plongement est donné par le système linéaire complet associé à $ H_{\!\boldsymbol{A}} $, on a  $(H_{\!\boldsymbol{A}}\cdot C )=3$ et  par conséquent 
$$d_{\boldsymbol{A}}= \big(E_{\!\boldsymbol{A}}\cdot C \big)=3\big(i(X_{\!\boldsymbol{A}})+1\big)=2\, r_{\!\boldsymbol{A}}+5=d_{exc}.\sk
$$ 
 
Pour construire des relations abéliennes de $\mathcal T_Z$, il convient d'étudier les formes finies sur $Z$.  Cela se fait au moyen d'un résultat classique d'adjonction pour lequel 
il nous faut supposer $Z$ lisse. Comme  $E_{\! \boldsymbol{A}}$ est très ample et $X_{\!\boldsymbol{A}}$ lisse, $Z$ générique dans $\lvert E_{\! \boldsymbol{A}} \lvert$ est lisse d'après le premier Théorème de Bertini. Sous cette hypothèse, on a une suite exacte 
\begin{equation}
\label{E:se}
0\rightarrow   K_{X_{\!\boldsymbol{A}}}  \longrightarrow 
K_{X_{\!\boldsymbol{A}}}( Z)
 \stackrel{{\rm Res}_Z}{\longrightarrow} K_Z\rightarrow 0
\end{equation}
où $K_{X_{\!\boldsymbol{A}}}( Z)=K_{X_{\!\boldsymbol{A}}}\otimes \mathcal O_{X_{\!\boldsymbol{A}}}(Z)$ 
 est le faisceau des $(r+1)$-formes  rationnelles sur $X_{\!\boldsymbol{A}}$ avec des pôles d'ordre 1 le long de $Z$ et où 
${\rm Res}_Z$ désigne le ``résidu de Poincaré'' (cf. \cite[p. 147]{GH}).  
\begin{align*}
\mbox{Par ailleurs, on a 
}\,  h^1\big( K_{X_{\!\boldsymbol{A}}} \big)= &\;  h^{r_{\!\boldsymbol{A}}}\big(
\mathcal O_{X_{\!\boldsymbol{A}}}
 \big) && \mbox{(par la dualité de Serre)} \\ 
= &\;  h^{0,r_{\!\boldsymbol{A}}}\big(X_{\!\boldsymbol{A}} \big) && \mbox{(par l'isom. de Dolbeault)} \\ 
= &\;  h^{r_{\!\boldsymbol{A}},0}\big(X_{\!\boldsymbol{A}} \big) && \mbox{(d'après la théorie de Hodge)} \\ 
= &\;  h^{0}\big( \Omega_{X_{\!\boldsymbol{A}}}^{r_{\!\boldsymbol{A}}} \big) && \mbox{(par l'isom. de Dolbeault)}.
\end{align*}

Comme $X_{\!\boldsymbol{A}}$ est rationnelle, on a  $h^0(\Omega_{X_{\!\boldsymbol{A}}}^{q})=0$ pour tout $q$ strictement positif   et donc $h^0(K_{X_{\!\boldsymbol{A}}})=h^1(K_{X_{\!\boldsymbol{A}}})=0$. 
Le premier morceau non-trivial de la suite longue de cohomologie associée à  (\ref{E:se}) nous donne donc un isomorphisme linéaire
$ H^0(X_{\!\boldsymbol{A}}, K_{X_{\!\boldsymbol{A}}}( Z) ) \simeq H^0(Z, K_{Z})$. 
%

 Or, par la définition même de $E_{\!\boldsymbol{A}}$, on a 
$K_{X_{\!\boldsymbol{A}}}( Z)=H_{\! \boldsymbol{A}}=[ \mathcal O_{\!X_{\!\boldsymbol{A}}}(1) ]$ et donc  
$h^0(K_Z)=h^0\big(\mathcal O_{\!X_{\!\boldsymbol{A}}}(1)\big)
=2\, r_{\!\boldsymbol{A}}+4$.
 D'après le Corollaire \ref{C:FinieAW}, cela implique que ${\rm rg}(\mathcal T_Z)\geq 2\, r_{\!\boldsymbol{A}} +4$. 
Vu que   par ailleurs ${\rm rg}(\mathcal T_Z)\leq \rho_{r_{\!\boldsymbol{A}},3}(d_{exc})=2r_{\!\boldsymbol{A}}+4$, on obtient que $\mathcal T_Z$ est de rang maximal.  
\sk 

Enfin, il est facile 
 de vérifier que 
$X_{\mathcal T_Z}=X_{\!\boldsymbol{A}}$, i.e. que 
la  variété de Blaschke du tissu $\mathcal T_Z$ n'est rien d'autre que  
$X_{\!\boldsymbol{A}}$.  Cette dernière  n'étant pas standard d'après la Proposition \ref{P:XJnonStandard}, on déduit du  
Théorème \ref{T:1} que $\mathcal T_Z$ n'est pas algébrisable au sens classique. On a obtenu le 
\bt
Si $Z$ est une hypersurface lisse du  système linéaire $\lvert  E_{\!\boldsymbol{A}}\lvert$, 
 alors  
 $(X_{\!\boldsymbol{A}}, \Sigma_{\!\boldsymbol{A}},Z)$ est admissible et définit  un $d_{exc}$-tissu algébrique exceptionnel  de type $(3\dim(\boldsymbol{A})+2,3)$ sur la famille de cubiques  $\Sigma_{\!\boldsymbol{A}}$. \et

L'énoncé analogue en supposant seulement $Z$ réduite et pas forcément lisse est très certainement encore valable,  voir Section \ref{S:ZpasLisse} plus loin. 


\subsection{Tissus  exceptionnels construits à partir de  ${\rm Seg}(\mathbb P^1\times Q)$} 
\label{S:WebExceptfromP1Q}
Soit $r$ un entier plus grand ou égal à 2.
On désigne par $Q$  une hypersurface quadrique lisse de  $ \mathbb P^{r+1}$ et 
on note $X$ le plongement de Segre ${\rm Seg}(\mathbb P^1\times Q)\subset \mathbb P^{2r+3}$.  Soit $H_{\mathbb P^1}$ (resp. $H_Q$), la classe de l'image réciproque de $\mathcal O_{\mathbb P^1}(1)$ (resp. $\mathcal O_Q(1)$) par la projection canonique $\pi_1:X\rightarrow \mathbb P^1$ (resp. $\pi_2:X\rightarrow Q$). 
On sait que le groupe de Picard de $X$ est librement engendré par $H_{\mathbb P^1}$ et $H_Q$ et que $H_{\mathbb P^1}+H_Q$ est très ample et correspond au plongement de Segre $X\subset \mathbb P^{2r+3}$ considéré ici. 

On pose alors 
$$
E_X= 3{H}_{\mathbb P^1}+(r+1)H_{Q}.
$$
 
 Soit $c:\mathbb P^1 \rightarrow C$, un paramétrage projectif  d'une courbe $C$ de la famille $\Sigma_3(X)$.  Si l'on suppose $C$ générale, on vérifie que $\pi_1\circ c:\mathbb P^1\rightarrow \mathbb P^1$ est un isomorphisme
 et donc $(C\cdot H_{\mathbb P^1})=1$.  Comme $\pi_2\circ c$ paramètre $\pi_2(C)$ qui est une conique dans $\mathbb P^{r+1}$, il vient $(C\cdot H_{Q})=2$.  On en déduit que 
 $( C\cdot E_X  \big)=2r+5$. 
 Comme $X$ est homogène, on a $X=X_{\rm adm}$ et donc toute hypersurface réduite $Z\in  \lvert  E_X\lvert $ est  telle que $(X,\Sigma_3(X),Z)$ est admissible et donc définit un tissu de type $(r,3)$ sur $\Sigma_3(X)$, noté $\mathcal T_Z$. \sk
 
 Supposons maintenant qu'en plus d'être réduite, $Z$ soit  lisse.  Par adjonction, on obtient un isomorphisme $H^0(K_X(Q))\simeq H^0(K_Z)$. Comme $K_X(Q)=\mathcal O_X(H_{\mathbb P^1}+H_{Q})$, il vient  $h^0(K_Z)=2r+4$, ce qui nous donne le 
\bt 
Si $Z$ est une hypersurface lisse du  système linéaire $\lvert E_{X}\lvert$
, alors  le triplet 
 $(X, \Sigma_3(X),Z)$ est admissible et définit  un tissu algébrique exceptionnel  de type $(r,3)$ sur la famille de cubiques $\Sigma_3({X})$.
  \et
  

\subsection{Une famille de tissus  exceptionnels de type $(r,4)$} 
\label{S:W3r+7(4,r)}  On fixe $r\geq 2$ et on reprend les notations de la Section  \ref{S:3H-2Q}.  Pour  construire des  tissus exceptionnels sur  
 $\Sigma_Q=\Sigma_5(X_Q)$, il nous faut d'abord obtenir des informations sur la géométrie de la variété $X_Q$ elle-même. 

\subsubsection{Quelques propriétés géométriques de $X_Q$} 
La variété $X_Q$ étant définie comme l'image de  $\varphi=\varphi_{\lvert 3H -2Q \lvert}: \mathbb P^{r+1}\dashrightarrow  \mathbb P^{3r+5}$, on va  commencer par résoudre les indéterminations de cette
application rationnelle.

On vérifie tout d'abord  que le lieu d'indétermination 
   de $\varphi$ est le sous-espace $\Pi=\langle Q\rangle$ engendré par $Q$. On considère alors 
 $\mu_1: X_1={\rm Bl}_\Pi(\mathbb P^{r+1})\rightarrow \mathbb P^{r+1}$, 
l'éclatement de $\mathbb P^{r+1}$ le long de $\Pi$. Le diviseur exceptionnel associé $E_1=\mu_1^{-1}(\Pi)$ s'identifie au projectifié du fibré normal de $\Pi$. Comme $N_{\Pi/\mathbb P^{r+1}}\simeq \mathcal O_\Pi(1)\oplus \mathcal O_\Pi(1) $, on en déduit  une identification naturelle  $E_1\simeq \Pi\times \mathbb P^1$.  Modulo celle-ci, la restriction 
de $\mu_1$ à $E_1$ s'identifie à la projection canonique $\Pi\times \mathbb P^1\rightarrow \Pi$ sur le premier facteur.  Il en découle que $Q_1=\mu_1^{-1}(Q)$ est isomorphe à $Q\times \mathbb P^1$.  

Soit $\varphi_1: X_1\dashrightarrow X_Q$ l'application birationnelle qui rend commutatif le  diagramme suivant: 
\begin{equation*}
    \xymatrix@R=0.7cm@C=1.9cm{  
    X_1   
       \ar@/^1pc/@{-->}[rd]^{\varphi_1\qquad }  \ar@{->}[d]_{\mu_1}
 &     \\ 
\mathbb P^{r+1} \ar@{-->}[r]_{\varphi \; \quad  }
& X_Q   \;  .
}
 \end{equation*}

On note $H_1$ la classe de l'image inverse par $\mu_1$ d'un hyperplan général de $\mathbb P^{r+1}$.
\begin{lemm}
 L'application $\varphi_1$ est induite par le système linéaire  $\mathcal L(\varphi_1)=\lvert 3H_1-E_1-Q_1\lvert$   
    des éléments de  $\lvert 3H_1-E_1 \lvert$  qui contiennent    $Q_1$. On a:
$$ \mathcal L(\varphi_1)=\varphi_1^{-1}\big(\lvert \mathcal O_{X_Q}(1)\lvert\big)
= \mathbb P H^0\Big(X_1,\mathcal O_{X_1}(3H_1-E_1)\otimes \mathcal I_{Q_1}\Big).
$$

En particulier, le schéma de base de $\varphi_1$ est $Q_1$. 
\end{lemm}
\begin{proof}  La transformée totale de $\lvert 3H-2Q_1\lvert$ par $\mu_1$ est un sous-système linéaire de $\lvert 3H_1\lvert$ dont la composante fixe est $E_1$. Il en découle que $\varphi_1$ est induite par un sous-système   linéaire de $\lvert 3H_1-E_1 \lvert$.\sk 

Soient  $(u,s)=(u_0,u_1,s_1,\ldots,s_r)$, les coordonnées homogènes sur $\mathbb P^{r+1}$ introduites 
 dans la   Section  \ref{S:3H-2Q}. On se place sur la carte affine $U\subset \mathbb P^{r+1}$ donnée par l'équation $s_r=1$. 
 Alors, sur un certain ouvert affine $U_1\subset \mu_1^{-1}(U)$ de $X_1$, 
 il existe des coordonnées affines  
 $(v,s)=(v_0,v_1,s_1,\ldots,s_{r-1})$  telles que la restriction de  $\mu_1$ à ${U_1}$ s'écrive 
\begin{equation}
\label{E:mu1surU1}
 (v,s)=(v_0,v_1,s_1,\ldots,s_{r-1})\longmapsto  (v_0,v_0v_1,s_1,\ldots,s_{r-1})=(u,s). 
 \end{equation}
 
  Les hypersurfaces cubiques de  ${\mathbb P^{r+1}}$ qui contiennent  $\Pi$  sont exactement celles qui sont découpées par les polynômes homogènes  de la forme  
$$
F(u,s)=R_3(u)+\sum_{j=1}^r P_2^j(u)s_j+ (c_0u_0+c_1u_1)P_2(s)
$$
où $R_3,P_2$ et les $P_2^j$ sont des polynômes homogènes  (de degré 3 et 2 respectivement) et où $c_0,c_1$  sont des constante complexes.   En utilisant \eqref{E:mu1surU1},  on obtient  que la transformée stricte   par $\mu_1$ 
 d'une telle  hypersurface cubique   correspond  exactement à l'hypersurface de $U_1$  découpée par l'équation 
$$
F_1(v,s)=v_0^2R_3(1,v_1)+ v_0\sum_{j=1}^r P_2^j(1,v_1)s_j+ (c_0+c_1v_1)P_2(s). 
$$

On vérifie immédiatement que $dF_x=0$ pour tout $x\in Q$ si et seulement si $q(s)$ divise $P_2(s)$
et que cela  équivaut à ce que $F_1(x_1)=0$ pour tout $  x_1\in Q_1$.  Cela signifie bien que $\mu_1^{-1}(\lvert 3H-2Q\lvert)=\lvert 3H_1-E_1-Q_1\lvert$ et démontre le lemme. 
\end{proof}

Le lemme ci-dessus implique que l'on va obtenir  une résolution lisse de $\varphi_1$ en éclatant $X_1$ le long de $Q_1$. On note $\mu_2: X_2={\rm Bl}_{Q_1}(X_1)\rightarrow  X_1$ cet éclatement. Alors il existe un morphisme $\varphi_2:X_2\rightarrow X_Q$ qui rend le diagramme suivant commutatif: 
\begin{equation*}
    \xymatrix@R=0.7cm@C=1.9cm{   
    X_2 \ar@{->}[d]_{\mu_2}     \ar@/^2pc/@{->}[rdd]^{\varphi_2\qquad\quad   } 
    \\
    X_1   
       \ar@/^1pc/@{-->}[rd]^{\varphi_1\qquad }  \ar@{->}[d]_{\mu_1}
 &    \\ 
\mathbb P^{r+1} \ar@{-->}[r]_{\varphi \; \quad  }
&   X_Q \;  .
}
\end{equation*}

De plus, de la propriété universelle de l'éclatement, on déduit que  
$$\varphi_2^*(\mathcal O_{X_Q}(1))=\mu_2^*(3H_1-E_1)-Q_2=3H_2-\varphi_2^*(E_1)-Q_2=3H_2-E_2-2Q_2,  $$
avec $H_2=\mu_2^*(H_1)$ et  $Q_2=\mu_2^{-1}(Q_1)$ ainsi que $E_2=\mu_2^{-1}(E_1)$.  Puisque $X_Q\subset \mathbb P^{3r+5}$ est linéairement normale, il en découle 
 que $\varphi_2$ est induite par 
 $\mathcal L(\varphi_2)=\lvert 3H_2-E_2-2Q_2 \lvert$. \sk 

Nous allons expliciter l'application $\varphi_2$ dans des coordonnées affines adaptées.  Pour cela, il est  pratique de supposer que les  $s_i$ ont été choisis de telle sorte que 
 la forme quadratique $q$ s'écrive $q(s)=s_1s_r-\sum_{i=2}^r s_i^2$, ce qui ne fait pas perdre en généralité.  On reprend les notations introduites dans la preuve du lemme précédent. 
 Sur l'ouvert affine $U_1$ de $X_1$,  on définit un nouveau système de coordonnées affines $(v,\sigma)=(v_0,v_1,\sigma_1,\ldots,\sigma_{r-1})$ en posant $\sigma_1=q(s_1,\ldots,s_{r-1},1)$ et  $\sigma_j= s_j$ pour $ j=2,\ldots,r-1$.
Dans ces nouvelles coordonnées,  $E_1\cap U_1$ et $Q_1\cap U_1$ ont respectivement pour équations $v_0=0$ et $v_0=\sigma_1=0$,  et les restrictions à $U_1$ des éléments de $\mathcal L(\varphi_1)$ sont les hypersurfaces découpées par les équations 
\begin{equation}
\label{E:F1surU1}
0= v_0^2R_3(v_1)+v_0\bigg[P_2^1(v_1)\big(\sigma_1+p(\sigma)\big)+ \sum_{j=2}^{r-1}P_2^j(v_1)\sigma_j+ P_2^r(v_1)\bigg]+(c_0+c_1v_1)\sigma_1,
\end{equation}
 où $P_2^1,\ldots,P_2^r$ et $R_3$  sont des polynômes non-homogènes, de degré 2 et 3 respectivement et où l'on a posé $p(\sigma)=\sum_{i=2}^{r-1}\sigma_i^2$. 
 
  Soient $(w_0,v_1,w_1,\sigma_2,\ldots,\sigma_{r-1})$ les coordonnées sur un certain  ouvert affine  $U_2\subset \mu_2^{-1}(U_1)$  telles   que  la restriction de $\mu_2$ à ${U_2}$
 s'écrive 
\begin{equation}
\label{E:mu2surU2}
 (w_0,v_1,w_1,\sigma_2,\ldots,\sigma_{r-1})\longmapsto  
  (w_0w_1,v_1,w_1,\sigma_2,\ldots,\sigma_{r-1})=(v,\sigma).
 \end{equation}
 
 Dans les coordonnées considérées sur $U_2$, les diviseurs $E_2\cap U_2$ et $Q_2\cap U_2$ sont respectivement découpés par les équations $w_0=0$ et $w_1=0$.  Au moyen de l'expression explicite \eqref{E:mu2surU2} de $\mu_2\lvert_{U_2}$, on détermine facilement l'équation de la transformée stricte par $\mu_2$ d'une hypersurface  découpée par une équation \eqref{E:F1surU1}. 
 On en déduit qu'une base de l'espace des  composantes de la restriction à $U_2$ de $\varphi_2: X_2\rightarrow \mathbb P^{3r+5}$  est donnée par l'ensemble des fonctions polynomiales suivantes
  \begin{equation}
 \label{E:varphi2}
  1\, ,\,v_1\, ,\,w_0 v_1^k \, ,\,w_0v_1^k \sigma_j\, ,\,  w_0v_1^k \big( w_1+p(\sigma) \big)\, ,\, w_0^2w_1 v_1^k\, ,\, w_0^2w_1 v_1^3, 
 \end{equation}
 où  $j=2,\ldots,r-1$ et $k=0,1,2$.
Au moyen des nouvelles variables  $x_0,\ldots,x_r$ liées aux précédentes par les relations 
\begin{equation}
\label{E:xrel}
x_0= w_0\, ,\,  x_1=  v_1
\, ,\, 
x_2= w_0\, \sigma_2\, , \, \ldots , \, x_{r-1}= w_0\, \sigma_{r-1}    
\, ,\,   x_r= w_0 \big( w_1+p(\sigma)\big) , 
\end{equation}
on vérifie que $\varphi_2(U_2)$ peut aussi être décrit comme l'image de l'application $\tilde\varphi_2: \mathbb C^{r+1}\rightarrow \mathbb P^{3r+5}$ dont les composantes sont 
 \begin{equation}
 \label{E:tildevarphi2}
 1\, ,\, x_1 \, ,\,     x_0x_1^k\, ,\, x_l x_1^k\, ,\, \big(x_0x_r-p(x)\big)x_1^k\, ,\,\big(x_0x_r-p(x)\big)x_1^3
  \end{equation}
avec $l=2,\ldots,r$ et $k=0,1,2$, où l'on a posé $p(x)=
\sum_{i=2}^{r-1}x_i^2$.   \medskip 

Nous sommes maintenant en mesure d'énoncer les propriétés de la variété $X_Q$ que nous utiliserons pour construire des tissus de rang maximal sur $\Sigma_Q$.  Pour simplifier,  
 on pose $X_3=X_Q$ dans l'énoncé suivant: 
\begin{prop}${}^{}$ 
\begin{enumerate}
\item[{\it 1.}]  L'image $E_3=\varphi_2(E_2)$ de $E_2$ par $\varphi_2$ est une droite de $\mathbb P^{3r+5}$ et  la restriction de $\varphi_2$ à ${E_2}$ s'identifie à la projection canonique $\Pi\times \mathbb P^1\rightarrow \mathbb P^1$ sur le second facteur.\sk 
\item[{\it 2.}]  Par restriction,  $\varphi_2$ induit un isomorphisme  $X_2\setminus E_2\simeq X_3\setminus E_3$.
\sk 
\item[{\it 3.}]  La variété $X_3$ est lisse et $\varphi_2$ est l'éclatement de $X_3$ le long de $E_3$.\sk  
\item[{\it 4.}] Si  $H_3$ et $Q_3$ désignent  les images directes de $H_2$ et $Q_2$ par $\varphi_2$, on a 
$$K_{X_3}=-(r+2)H_3+2Q_3\qquad 
\mbox{ et }\qquad \big[\mathcal O_{X_3}(1)\big]=3H_3-2Q_3.
$$ 
\end{enumerate}
\end{prop}
\begin{proof}
La variété $X_2$ est recouverte par un nombre fini de cartes affines dans lesquelles $\varphi_2$  s'exprime essentiellement de la même façon que dans la carte particulière $U_2$  considérée ci-dessus. Pour établir les points {\it 1.},  {\it 2.}  et {\it 3.}, on peut donc se borner à démontrer les énoncés correspondants lorsqu'on se restreint à $U_2$. 
En gardant ce fait à l'esprit, \medskip 
\begin{itemize}
\item  on déduit les points {\it 1.} et {\it 2.} de l'expression explicite \eqref{E:varphi2} des composantes de $\varphi_2$ sur $U_2$ (modulo des vérifications faciles laissées au lecteur);\medskip 
\item  on obtient la première partie du point {\it 3.} en observant que $\tilde\varphi_2: \mathbb C^{r+1}\rightarrow \mathbb P^{3r+5}$ est un plongement puisque 
$x_0,\ldots,x_r$ en sont des composantes d'après  \eqref{E:tildevarphi2}. Que $\varphi_2$ soit l'éclatement de $X_3$ le long de $E_3$ découle alors des relations \eqref{E:xrel}.  
\medskip 
\end{itemize}

Comme $\varphi_2$ contracte $E_2$, le point {\it 4.} est une conséquence facile du fait que $\varphi_2$ est induite par le système linéaire $\lvert 3H_2-E_2-2Q_2\lvert$ (voir plus haut) et que 
$$K_{X_2}=(\mu_1\circ \mu_2)^*(K_{\mathbb P^{r+1}})+\mu_2^*(E_1)+Q_2=-(r+2)H_2+E_2+2Q_2 
$$
 d'après \cite[Exercise II.8.5.(b)]{hartshorne}, vu que $\mu_2^*(E_1)=E_2+Q_2$.
\end{proof}

\begin{rem*} {\rm 
La quadrique  $Q$ est irréductible si et seulement si $r>2$. Dans ce cas, le groupe de Picard de $X_Q$ est librement engendré par les classes $  H_3 $ et $Q_3$. Par contre, $Q$ est l'union de deux points distincts lorsque $r=2$. Dans ce cas,  $Q_3$ n'est pas irréductible mais est la somme  de deux diviseurs irréductibles $Q_3'$ et $Q_3''$ qui s'intersectent le long de la droite $E_3$ et le groupe de Picard de $X_Q$ est libre, de rang 3,  engendré  par les classes $H_3,Q_3'$ et $Q_3''$.}
\end{rem*} 

On se donne maintenant  une cubique gauche $C_0$ dans $\mathbb P^{r+1}$  qui correspond à une quintique rationnelle générique $C=\varphi_2(C_0)$ de la famille $\Sigma_Q$. Vu la fin de la Section \ref{S:3H-2Q} et par généricité,  on sait que $C_0$ intersecte le sous-espace projectif $\Pi$ transversalement   en deux points distincts $a, a'$  de $Q$.  La  transformée stricte $C_1=\mu_1^{-1}(C_0)$   de $C_0$ sur $X_1$ intersecte $E_1$ en deux points  de $Q_1$, et cela de façon transverse 
 puisque ni $a$ ni $a'$ n'est un point d'inflexion de $C_0$  (qui n'en a pas, vu que c'est une cubique gauche).  Par conséquent, on obtient que la transformée stricte $C_2$ 
 de $C_1$ par $\mu_2$
  intersecte $Q_2$ en deux points distincts et 
 ne rencontre pas $E_2$.  
D'après  le point {\it 2.} de la proposition ci-dessus, cela implique que  $(C\cdot Q_3)=2$.   

Des arguments précédents, on peut  également déduire que  l'ensemble  des points admissibles de $X_Q$ est exactement $X_Q\setminus E_3$.  
\medskip 
 
 Nous avons maintenant à disposition tout ce dont nous avons   besoin pour construire des tissus algébriques d'incidence de rang maximal sur la famille de quintiques rationnelles $\Sigma_Q$.  \sk
 
    \subsubsection{Tissus de rang maximal de type $(r,4)$sur $\Sigma_Q$}
 
 On pose  
 $$
 E_{X_Q}=(r+5)H_3-4Q_3 \in {\rm Pic}(X_Q). 
 $$
 
  Soit $Z$  une hypersurface réduite du système linéaire $\lvert 
   E_{X_Q} \lvert$. Elle rencontre $(X_Q)_{\rm adm}=X_Q\setminus E_3$,   le triplet 
  $ (X_Q,\Sigma_Q,Z)$ est donc admissible d'après la Proposition \ref{P:XVSigmaAdm}
  et définit un tissu de type $(r,4)$ noté $\mathcal T_Z$.  
  Pour  $C\in \Sigma_Q$ générale, on a  $(C\cdot H_Q)=3$ et $(C\cdot Q_2)=2$ et donc 
  $(C\cdot  E_{X_Q})=3(r+5)-8$  . On en déduit que 
  $\mathcal T_Z$ est d'ordre 
     $d_{exc}(r,4)$. 
  \sk

On peut écrire $ E_{X_Q}=6H_3-4Q_3+(r+5-6)H_3$. Comme $6H_3-4Q_3= \mathcal O_{X_Q}(2)$ et puisque $(r+5-6)H_3$ est effectif vu que $r>1$, la classe $ E_{X_Q}$ est très ample. Par conséquent,  l'élément général $Z$ de $\lvert E_{X_Q} \lvert$ est lisse.  Sous cette dernière hypothèse, on a par adjonction 
$$h^0(K_Z)=h^0\big(K_{X_Q}(Z)\big)=h^0\big(\mathcal O_{X_Q}(1)\big) =3r+6=\rho_{r,4}(3r+7). $$   

On utilise alors le Corollaire  \ref{C:FinieAW}, pour  en déduire  le 
\bt 
Si $Z$ est un élément lisse de  $\lvert E_{X_Q}\lvert$ 
 alors le triplet 
 $(X_Q,\Sigma_{Q},Z)$ est admissible et définit  un $d_{exc}$-tissu algébrique exceptionnel  de type $(r,4)$ sur la famille de quintiques rationelles $\Sigma_{Q}$.
  \et


\subsection{Tissus  exceptionnels de type $(2,6)$ construits à partir de  $v_3(\mathbb P^3)$}
Soit  $S  $
 une surface réduite  de degré 7 dans $\mathbb P^3$.   On note ici
$$\Sigma={{SL_4(\mathbb C)}/{SL_2(\mathbb C)}}\subset {\rm Hilb}^{3t+1}(\mathbb P^3)$$
 la  famille des courbes cubiques gauches  de $\mathbb P^3$.  Un élément général  $C$ de  $\Sigma$  intersecte $S$ transversalement en  $21$ points d'après le Théorème de Bézout. Le triplet $(\mathbb P^3,\Sigma,S)$ est admissible et définit donc un $21$-tissu 
de type $(2,6)$  sur $\Sigma$, qu'on note $\mathcal T_S$. 
\smallskip

Supposons que $S$ soit lisse. Par adjonction, on a $K_S=\mathcal O_{S}(3)$ et donc $h^0(S,K_S)=20$. D'après les résultats de la Section \ref{S:Abel}, on obtient que  ${\rm rg}(\mathcal T_S)=20=\rho_{2,6}(21)$.  On en déduit le 
\bt
 Soit   $S$ une surface lisse de degré $7$ dans $ \mathbb P^3$.  Le triplet 
   $(\mathbb P^3,\Sigma,S)$ est admissible et définit un 21-tissu  algébrique exceptionnel de type $(2,6)$ sur la famille  des cubiques gauches de $\mathbb P^{3}$.  
 \et

Signalons ici que dans les ``{\it Notes added in proof}''de \cite{Griffiths}, Griffiths discutait déjà de l'annulation des traces,  par rapport à la famille des  cubiques gauches de $\mathbb P^3$, des 2-formes différentielles holomorphes  sur une surface quartique lisse.  Mais rien ne laissait deviner qu'il   
 suffisait de considérer à la place des surfaces de degré 7 et non pas 4 pour obtenir des tissus de rang maximal d'un type nouveau...


\subsection{Tissus  exceptionnels de type $(2,5)$ construits à partir de  $\tau_x(v_3(\mathbb P^3))$}
  Soit $\mu: X{\rightarrow}\mathbb P^3$ l'éclatement de $\mathbb P^3$ en  $x$.  Son groupe de Picard   est librement engendré par les classes 
   de l'image réciproque par $\mu$ d'un plan général de $\mathbb P^3$ 
  et du diviseur exceptionnel $\mu^{-1}(x)$, notées $H$ et $E$ respectivement.

   On vérifie que la classe $3H-2E\in {\rm Pic}(X)$ est très ample et que l'image de $X$ par le plongement associé est la projection tangentielle $\tau_x(v_3(\mathbb P^3))$.  D'autre part,  on a $K_X=\mu^*(K_\mathbb P^3)+2E=-4H+2E$.  On pose alors 
   $$E_X=-K_X+(3H-2E)=7H-4E\in {\rm Pic}(X). $$  
    
   Soit $S$ une surface réduite du système linéaire $  | E_X |$. 
Si $C$ est une cubique gauche  de $\mathbb P^3$ qui passe par $x$,  en général 
$C'=\mu^{-1}(C)$ intersecte $S$ transversalement en 
 $$(C'\cdot E_X)=7 (C'\cdot H)-4(C'\cdot E)=7\cdot 3-4\cdot 1=17$$ points distincts.  Par conséquent  $(X,\Sigma_3(X),S)$ est admissible et le 17-tissu associé  est de type $(2,5)$.  
 Puisque $E_X$ est très ample, $S\in |E_X|   $ générale est lisse. Dans ce cas, par adjonction, on a     $h^0(S,K_S)=16$.  
  De la Section \ref{S:Abel}, on déduit alors le 
 
 \bt 
Si $S\in \lvert  E_{X}\lvert     $ est  lisse,  alors   $({\rm Bl}_x(\mathbb P^3),\Sigma_3({{\rm Bl}_x(\mathbb P^3)}), S)$ est un triplet admissible qui définit  un 17-tissu algébrique exceptionnel de type $(2,5)$
 sur $\Sigma_3({{\rm Bl}_x(\mathbb P^3)})$.   \et

\subsection{Tissus  exceptionnels de type $(2,4)$ construits à partir de  $\tau_{xy}(v_3(\mathbb P^3))$}
Ce cas correspond au cas $r=2$ de la famille de tissus de type $(r,4) $ construite  Section \ref{S:W3r+7(4,r)} ci-dessus. 


\section{Remarques et commentaires}
\label{S:Final}
Pour finir, nous faisons plusieurs remarques qui concernent tout d'abord les résultats obtenus auparavant et certaines conclusions que l'on peut en tirer a priori sans trop de difficulté. Ensuite, nous disons quelques mots de  ce qu'il faudrait arriver à comprendre pour aboutir à une classification complète des tissus algébriques exceptionnels.

\subsection{\bf Tissus associés à un triplet admissible ${(X,\Sigma_{2n-3}(X),Z)}$ avec  $Z$ réduite}
\label{S:ZpasLisse}
Soit $X$  une variété  non-standard lisse dans la classe $\mathcal X_{r+1,n}(2n-3)$ considérée plus haut. 
On a défini une certaine classe  ample $E_X\in {\rm Pic}(X)$ telle que pour $Z\in \lvert  E_X \lvert$ lisse, le triplet $(X,\Sigma_{2n-3}(X),Z)$ est admissible et définit un tissu algébrique exceptionnel. 
Si $Z$ est seulement supposée réduite, le triplet correspondant est encore admissible et il est  naturel de  penser que le $d_{exc}$-tissu de type $(r,n)$ associé $\mathcal T_Z=\mathcal T_{(X,\Sigma_{2n-3}(X),Z)}$ est encore de rang maximal. \sk

En  effet, dans ce cas  on dispose  également  d'une suite exacte 
courte  $ 0\rightarrow   K_{X}  \rightarrow 
K_{X}( Z) {\rightarrow} \omega^r_{Z}\rightarrow 0
$ (cf. \cite[p. 228]{Aleksandrov}) qu'on peut utiliser, par un calcul tout à fait analogue à celui fait  dans le cas  lisse, pour établir que 
\begin{equation}
\label{E:H0omegaVr}
h^0(\omega_{Z}^r)=h^0(X,K_X+E_X)=\rho_{r,n}\big(d_{exc}\big).
\end{equation}

 Pour obtenir que  $\mathcal T_{Z}$ est de rang maximal, il faudrait par exemple disposer d'une version du 
  Théorème  \ref{T:AbelGriffiths} dans laquelle on aurait  remplacé les formes finies sur  $Z$  par les formes abéliennes sur cette hypersurface. Comme on l'a dit, il est naturel  de penser 
qu'une telle version existe.\smallskip 

Une autre approche pour démontrer que $\mathcal T_{Z}$  est de rang maximal quand $Z$ est seulement supposée réduite pourrait reposer sur des arguments de déforma\-tion. Comme $X$ est lisse et $ E_X$ très ample, il découle du théorème de Bertini que toute hypersurface dans  $\lvert  E_X\lvert $ est limite d'éléments lisses 
de ce système linéaire.   Par conséquent,  le tissu  $\mathcal T_{Z}$ peut être  vu comme une limite,  en un sens naturel,  de  tissus de rang maximal.  On obtiendrait que $\mathcal T_{Z}$ est lui aussi de rang maximal en montrant que la maximalité du rang est une condition fermée pour les familles analytiques de tissus  de type $(r,n)$. 
Si cela semble  intuitivement  aller de soi, ce dernier point   mériterait d'être démontré de façon rigoureuse.  

\subsection{\bf Sur l'équivalence des tissus  exceptionnels  construits précédemment} Dans cet article, nous avons construit des tissus algébri\-ques exceptionnels 
à partir de certains triplets admissibles $(X,\Sigma_{2n-3}(X),Z)$ où $X$ est une variété  non-standard lisse d'une classe $\mathcal X_{r+1,n}(2n-3)$ et $Z$ un élément du systeme linéaire $\lvert E_X\lvert$ associé à une certaine classe $E_X\in {\rm Pic}(X)$. 
En utilisant le fait que  ces tissus $\mathcal T_Z=\mathcal T_ {(X,\Sigma_{2n-3}(X),Z)}$  sont canoniques, on montre facilement que pour $Z,Z'\in \lvert  E_{X}\lvert$ réduites, les tissus $\mathcal T_Z$ et $\mathcal T_{Z'}$ sont (analytiquement) équivalents si et seulement si les hypersurfaces $Z$ et $Z'$ sont congruentes sous l'action du  groupe   des automorphismes projectifs de $X$, noté ${\rm Aut}(X)$.   \sk 

Soit $ \lvert E_{X}\lvert^{\rm o} $ le complémentaire  du fermé de Zariski formé des hypersurfaces non réduites.  De ce qui précède, il découle que l'espace des modules des classes d'équivalences analytiques des tissus $\mathcal T_Z$ s'identifie à l'espace des orbites de  $\lvert E_{X}\lvert^{\rm o} $ sous l'action de ${\rm Aut}(X)$.  Quand $X\in \mathcal X_{r+1,3}(3)$ est lisse et n'est pas un produit, 
${\rm Aut}(X)$ est  l'un des groupes de Lie complexes simples   $Sp_6(\mathbb C), SL_6(\mathbb C), SO_{12}(\mathbb C)$ ou $E_7$. On se retrouve donc ramené dans ce cas à un problème très classique (mais pas forcément facile à résoudre) de la  théorie des représentations des groupes de Lie complexes simples.

\subsection{\bf Sur la classification des tissus algébriques exceptionnels} 
Dans les sections précédentes, nous avons expliqué comment construire des tissus algébri\-ques exceptionnels
 à partir de variétés non standards des classes $\mathcal X_{r+1,n}(2n-3)$. 
Auparavant, nous avons montré que tout tissu de rang maximal non-algébrisable (au sens classique) s'obtient de cette façon. Ces résultats laissent ouvert le problème de la classification des tissus algébriques exceptionnels.   
 \sk 
 
 Vu le Théorème \ref{T:1}, cette question se ramène à la détermination de certaines  hypersurfaces des variétés non-standards des classes $ \mathcal X_{r+1,n}(2n-3)$ et donc, avant cela, passe par la classification des variétés de ce type.  C'est un horizon intéressant qui nous semble abordable, plus particulièrement dans le cas $n=3$.
 
  En effet,  d'après \cite{PR-XJC}, les variétés non-standards des classes 
 $\mathcal X_{r+1,3}(3)$ sont exactement les courbes cubiques $X_J$ sur les algèbres de Jordan $J$ de rang 3 (voir Section \ref{S:jordancubicXJ}). Par conséquent, 
 classer les $(2r+5)$-tissus  exceptionnels de type $(r,3)$ 
 demanderait de savoir décrire les hypersurfaces d'une courbe cubique $X_J$ arbitraire.  
 Ces variétés étant singulières en général, il faudrait pour cela  être en mesure de 
 décrire le groupe de Picard de $X_J$ ou de l'une de ses désingularisations. 
 \sk 
 
Le fait que le radical d'une algèbre de Jordan soit résoluble laisse entrevoir une façon naturelle de construire une désingularisation $\mu: X'_J\rightarrow X_J$ via une suite d'éclatements successifs qui ferait apparaître $X'_J$  comme l'espace total d'une tour de fibrés projectifs au dessus les uns des autres et dont la base serait la partie semi-simple et donc lisse  $X_{\!J_{\rm ss}}$ de $X_J$ (voir \cite{PR-XJC2} pour plus de détails). \`A  partir d'une telle construction, via des résultats classiques de géométrie algébrique, on pourrait  décrire ${\rm Pic}(X'_J)$ ainsi que  la classe correspondant au morphisme $X'_J\rightarrow  \mathbb PZ_2(J)$  (obtenu en composant $\mu$ et l'inclusion $X_J\subset \mathbb PZ_2(J)$) et aussi la classe canonique $K_{X'_J}$, c'est-à-dire décrire tout ce dont on a besoin pour construire et classifier des tissus exceptionnels. 
 On conjecture qu'il existe une classe $ E_{X_J'}\in {\rm Pic}(X_J')$ explicite et unique  telle qu'un triplet $(X_J,\Sigma_{3}({X_J}),Z)$ est admissible et définit un tissu algébrique exceptionnel si et seulement la transformée stricte de $Z$ par  $\mu$ est réduite et appartient au  système linéaire  $\lvert E_{X_J'} \lvert$. 
 \sk 
 
 Il est clair qu'un travail conséquent reste à fournir avant   que la stratégie présentée ci-dessus  aboutisse éventuellement. 
 Le premier cas à considérer est bien sûr celui dans lequel la variété $X$ d'une classe $\mathcal X_{r+1,3}(3)$ est lisse.  Ce qui est attendu se produit bien dans ce cas particulier  puisqu'on a le 
 \bt
  \label{T:W3XWlisse}
Soit  $\mathcal T$ un $(2r+5)$-tissu de rang maximal exceptionnel de type $(r,3)$.  Si  
  $X_{\mathcal T}$ est lisse, c'est une des variétés  de  la T{\small{ABLE}} 1 et   $Z_{\mathcal T}
  \in  \lvert E_{X}\lvert^{\rm o} $.
  De plus, 
    la trace induit un isomorphisme linéaire entre l'espace des $r$-formes abéliennes sur $Z_{\mathcal T}$ et l'espace des relations abéliennes du modèle canonique 
    ${\mathcal T}_{can}={\mathcal T}_{(X_{\mathcal T}, \Sigma_3(X_{\mathcal T}), Z_{\mathcal T})}$ de $\mathcal T$.     
\et
\begin{proof}
Sans perdre en généralité, on peut supposer que $\mathcal T$ est canonique, {\it i.e.} qu'il existe un triplet admissible $(X,\Sigma_3(X),Z)$  tel que $\mathcal T=\mathcal T _{(X,\Sigma_3(X),Z)}$
où  $Z$ est une hypersurface réduite de la variété $X$ qui est un élément lisse non-standard de la classe $ \mathcal X_{r+1,3}(3)$. 

Comme $\mathcal T$ est un $(2r+5)$-tissu, on doit avoir $(C\cdot Z)=2r+5$ pour toute courbe admissible $C\in \Sigma_3({X})$ non-incluse dans $Z$. Quand $X$ est lisse et n'est pas isomorphe au  produit de $\mathbb P^1$ avec une quadrique lisse, c'est l'une des quatre courbes cubiques $X_{\! \boldsymbol A}$  décrites 
 dans la Section \ref{S:Les quatres cas simples}.  Comme on a  ${\rm Pic}(X)\simeq \mathbb Z $ 
dans ce cas, on obtient immédiatement que $Z\in \lvert E_X \lvert^{\rm o}$.  \sk 

Supposons que  $X={\rm Seg}(\mathbb P^1\times Q)$ où $Q$ est une hyperquadrique lisse. 
Comme ${\rm Pic}(X)$ est librement engenré par $H_{\mathbb P^1}$ et $H_Q$ (voir Section \ref{S:WebExceptfromP1Q}),  il existe deux entiers   $a,b$ tels que $Z\in \lvert aH_{\mathbb P^1}+bH_Q\lvert$.  Ils  sont uniquement déterminés par $Z$ et sont positifs ou nuls puisque la classe 
$aH_{\mathbb P^1}+bH_Q$ est effective. En utilisant des arguments similaires à ceux utilisés ci-dessus, on déduit qu'ils vérifient (1)  $a+2b=2r+5$ et (2) $h^0((a-2)H_{\mathbb P^1}+(b-r)H_Q)\geq  2r+4$.  De (1), on déduit que $(a,b)=(3+2k,r+1-k)$ pour un certain entier $k\in \mathbb Z$.  En injectant cela dans l'inégalité (2), on obtient facilement que nécessairement $k=0$ et donc  $Z\in  \lvert  E _{\mathbb P^1\times Q}\lvert^{\rm o}$.\sk 

Dans tous les cas, puisque $X$ est homogène, on a $X_{\rm adm}=X$. Du point {\it (vii)} 
de  la sous-section {\it \ref{SS:141}}, il vient que les relations abéliennes de $\mathcal T$ s'obtiennent toutes à partir d'un sous-espace vectoriel de $H^0(Z,\omega_Z^r)$.  Comme 
$h^0(Z,\omega_Z^r)=2r+4={\rm rg}(\mathcal T)$ d'après \eqref{E:H0omegaVr} plus haut,  on obtient que la trace induit une identification linéaire $H^0(Z,\omega_Z^r)\simeq \mathfrak A(\mathcal T)$, ce qui termine la démonstration.
\end{proof}

\end{document}